\documentclass{amsart}

\usepackage{amsmath,amsthm,amsfonts,amssymb,amscd,amsbsy,dsfont,mathrsfs,hyperref}
\usepackage[all]{xy}

\newcommand{\dd}{\mathrm{d}}

\newcommand{\g}{\mathtt{g}}

\newcommand{\Ric}{\mathrm{Ric}}
\newcommand{\R}{\mathds{R}}
\newcommand{\Z}{\mathds{Z}}
\newcommand{\C}{\mathds{C}}
\newcommand{\Hr}{\mathds{H}}
\newcommand{\Ca}{\mathds{C}\mathrm a}

\newcommand{\Ad}{\mathrm{Ad}}

\newcommand{\spec}{\mathrm{Spec}}

\newcommand{\Emb}{\mathrm{Emb}}
\renewcommand{\H}{\mathcal{H}}

\newcommand{\G}{\mathsf G}
\newcommand{\K}{\mathsf K}
\newcommand{\Hs}{\mathsf H}
\newcommand{\GG}{\mathscr G}

\newcommand{\Sp}{\mathsf{Sp}}
\newcommand{\SU}{\mathsf{SU}}
\newcommand{\U}{\mathsf{U}}

\newcommand{\SO}{\mathsf{SO}}
\newcommand{\Spin}{\mathsf{Spin}}

\hypersetup{
    pdftoolbar=true,
    pdfmenubar=true,
    pdffitwindow=false,
    pdfstartview={FitH},
    pdftitle={Delaunay-type hypersurfaces in cohomogeneity one manifolds},
    pdfauthor={R. G. Bettiol and P. Piccione},
    pdfsubject={Primary: 53C42, 58E09;  Secondary: 53A10, 58D10, 58D19, 58J55},
    pdfkeywords={}
    pdfnewwindow=true,
    colorlinks=true, 
    linkcolor=blue,
    citecolor=blue,
    urlcolor=blue,
}

\newtheorem{theorem}{Theorem}[]
\newtheorem{lemma}[theorem]{Lemma}
\newtheorem{proposition}[theorem]{Proposition}
\newtheorem{corollary}[theorem]{Corollary}

\newtheorem*{mainthm}{Theorem}

\theoremstyle{definition}
\newtheorem{definition}[theorem]{Definition}

\theoremstyle{remark}
\newtheorem{remark}[theorem]{Remark}
\newtheorem{example}[theorem]{Example}

\title{Delaunay-type hypersurfaces in cohomogeneity one manifolds}
\author{Renato G. Bettiol \and Paolo Piccione}

\allowdisplaybreaks
\numberwithin{equation}{section}
\numberwithin{theorem}{section}

\address{\begin{tabular}{lll}
University of Pennsylvania & &Universidade de S\~ao Paulo \\
Department of Mathematics & & Departamento de Matem\'atica \\
209 South 33rd St  & & Rua do Mat\~ao, 1010 \\
Philadelphia, PA, 19104-6395, USA & & S\~ao Paulo, SP, 05508-090, Brazil\\
\emph{E-mail address}: {\tt rbettiol@math.upenn.edu} & & \emph{E-mail address}: {\tt piccione@ime.usp.br}
\end{tabular}
}

\date{July 13, 2015}
\thanks{The first named author is partially supported by the NSF grant DMS-0941615, USA. The second named author is supported by Fapesp and CNPq, Brazil.
}


\begin{document}
\begin{abstract}
Classical Delaunay surfaces are highly symmetric constant mean curvature (CMC) submanifolds of space forms. We prove existence of Delaunay-type hypersurfaces in a large class of compact manifolds, using the geometry of cohomogeneity one group actions and variational bifurcation techniques. Our construction specializes to the classical examples in round spheres, and allows to obtain Delaunay-type hypersurfaces in many other ambient spaces, ranging from complex and quaternionic projective spaces to Kervaire exotic spheres.
\end{abstract}

\maketitle

\section{Introduction}
Let $M$ be a Riemannian manifold and $N\subset M$ be a hypersurface that is the boundary of an open bounded subset of $M$. It is a classical fact in Riemannian geometry that $N$ has constant mean curvature (CMC) if and only if it has stationary area among variations that preserve the enclosed volume. The value $\H$ of its mean curvature is precisely the Lagrange multiplier of this constrained isoperimetric variational problem. Recall that the \emph{mean curvature} $\H$ of a submanifold is the norm of its \emph{mean curvature vector} $\vec H$, which is the trace of the second fundamental form. A submanifold (of any codimension) for which $\H\equiv0$  is called \emph{minimal}.

Among the fundamental objects in the theory of CMC hypersurfaces are the so-called \emph{Delaunay surfaces}, which are families of rotationally symmetric CMC surfaces in $\R^3$. These families consist of round spheres, cylinders, unduloids, catenoids and nodoids. They were classified in 1841, when Delaunay~\cite{del} ingeniously observed that a rotationally symmetric surface in $\R^3$ has CMC if and only if its profile curve is a \emph{roulette} of a conic section. More precisely, if a conic section rolls without slipping along a line, the rotation around this line of the curve traced by one of its foci is a CMC surface. Conversely, all rotationally symmetric CMC surfaces in $\R^3$ are obtained this way. Unduloids, catenoids and nodoids are the surfaces obtained by performing this construction with an ellipse, a parabola and a hyperbola, respectively. These surfaces play a fundamental role in the study of general CMC surfaces, e.g., it is known that each end of an embedded CMC surface with finite topology converges exponentially to the end of some Delaunay surface \cite{kks}.

Similar constructions of rotationally symmetric CMC hypersurfaces exist in $S^n$, $\R^n$ and $\mathds H^n$. In particular, those in $S^n$ are compact and can be understood as bifurcating branches from the family of CMC Clifford tori, see \cite{ap,mpnod}. As suggested by Pacard~\cite{pacard}, ``it is then a natural question to investigate the existence of Delaunay-type CMC hypersurfaces in any compact Riemannian manifold'', which ``despite some partial results, remains completely open''.

The purpose of this paper is to generalize the above mentioned bifurcation approach, proving existence of \emph{Delaunay-type hypersurfaces} on a large class of highly symmetric compact manifolds. Roughly, these are families of CMC hypersurfaces that bifurcate from a given family of homogeneous CMC hypersurfaces, partially preserving symmetries.
We briefly point out that bifurcation of CMC hypersurfaces has been recently studied by various authors, e.g., \cite{ap,kgb,he,mpnod}. As is the case in other geometric problems, the presence of many symmetries greatly simplifies the situation. More precisely, we deal with \emph{cohomogeneity one} manifolds, which are manifolds $M$ that support an isometric action of a Lie group $\G$ with an orbit of codimension one; or equivalently, such that the orbit space $M/\G$ is one-dimensional.

Cohomogeneity one manifolds constitute an important class that generalizes homogeneous spaces and has received great attention in the past two decades, see \cite{book,aa,gzricci,gzannals,gwz,straume,vz}. This class is also quite large, including examples ranging from compact rank one symmetric spaces to Kervaire exotic spheres. Notice that for a general isometric action of $\G$ on $M$, the mean curvature vector field of each orbit $\G(x)$ is parallel, due to its $\G$-equivariance. In particular, if $\G(x)\subset M$ has codimension one, then $\G(x)$ is an embedded CMC homogeneous hypersurface. If $M$ has cohomogeneity one, interior points of $M/\G$ correspond to principal $\G$-orbits in $M$, i.e., orbits whose isotropy group $\Hs$, called \emph{principal isotropy}, is the smallest possible (hence $\G/\Hs$ has codimension one). Thus, there is a $1$-parameter family $x_t(\G/\Hs)$ of CMC homogeneous hypersurfaces of $M$, parametrized by the interior of $M/\G$. Boundary points of $M/\G$ correspond to nonprincipal orbits, that are called \emph{exceptional} if they have the same dimension as $\G/\Hs$, and \emph{singular} otherwise. Because they are isolated, any singular orbit $S$ on $M$ is a minimal submanifold (see \cite{book,hsiang,hl}), on which the principal orbits \emph{condense}. In other words, any tubular neighborhood of $S$ in $M$ is foliated by CMC homogeneous hypersurfaces (the principal orbits $x_t(\G/\Hs)$), which are precisely geodesic tubes around $S$, i.e., consist of points at fixed distance from $S$. Moreover, the value of the mean curvature of these CMC geodesic tubes goes to $+\infty$ as they condense on $S$, see \eqref{eq:sxtblowsup}.

The above observations bring us to the second question that motivates the present work. Mahmoudi, Mazzeo and Pacard~\cite{mmp} proved that a geodesic tube of almost any small radius around a nondegenerate minimal submanifold $S$ can be deformed to yield a CMC hypersurface. The radii for which this deformation fails, which they call \emph{resonant radii}, ``should correspond to other families of CMC hypersurfaces bifurcating from this main `tubular' family''. In the case $\dim S=1$, previously studied by Mazzeo and Pacard~\cite{mp}, there are explicit degenerate examples for which the ``elements of the bifurcating families have undulations modeled on Delaunay surfaces, but the geometric picture for $\dim S>1$ is unknown''.

As observed above, on any cohomogeneity one manifold $M$ with a singular orbit $S$, the family of principal orbits that foliates a tubular neighborhood of $S$ provides an explicit example of CMC geodesic tubes around a minimal submanifold $S$ that fill up the \emph{entire} 
neighborhood. Our construction of Delaunay-type hypersurfaces on $M$ is based on obtaining infinitely many bifurcation instants for this family, exactly at its resonant radii. A common phenomenon in bifurcating branches is break of symmetry, and in the present work the crucial point is that \emph{partial symmetry preservation} occurs at every bifurcation. This yields our Delaunay-type hypersurfaces, which are infinitely many families of new hypersurfaces with (large) constant mean curvature, accumulating at geodesic tubes around $S$. Moreover, they are invariant under a symmetry group $\GG$ related to the symmetries of $S$.
In this way, with the aid of a large symmetry group, we provide further insight on the geometric picture of the bifurcating families from the tubular family around a minimal submanifold of \emph{any} (positive) dimension.

With such motivations in mind, let us state our main result. Unless otherwise mentioned, $\G$ and $M$ will always be assumed to be smooth and compact, and the isometric $\G$-action will be assumed smooth and almost effective (i.e., the ineffective kernel is finite). Furthermore, we work with $\G$-invariant metrics on $M$ that we call \emph{adapted} near $S$, see Definition~\ref{def:adapted}.

\begin{mainthm}
Let $M$ be a cohomogeneity one $\G$-manifold with a singular orbit $S=\G/\K$, $\dim S\geq1$, and principal isotropy $\Hs$. Assume that either $\Hs\triangleleft\K$ or $\K\triangleleft\G$ and the metric on $M$ is $\K$-invariant and adapted near $S$. Then there is a sequence $\H_q\in\R$ such that $\H_q\to+\infty$ and for each $\H_q$ there are infinitely many embedded Delaunay-type hypersurfaces in $M$ diffeomorphic to $\G/\Hs$, with constant mean curvature arbitrarily close to $\H_q$. Furthermore, these families of CMC hypersurfaces are not $\G$-orbits, but invariant under a certain $\K$-action, and condense on $S$ as $q\to+\infty$.
\end{mainthm}

There is a subtle difference in the meaning of \emph{Delaunay-type hypersurface} depending on which of the normality conditions  $\Hs\triangleleft\K$ or $\K\triangleleft\G$ is satisfied. On the one hand, if $\Hs\triangleleft\K$, the infinitely many bifurcating branches of CMC hypersurfaces issuing from the family of principal orbits near $S=\G/\K$ are invariant under the subaction of $\K$ on $M$, i.e., the restriction of the $\G$-action to $\K$. On the other hand, if $\K\triangleleft\G$, then such branches of CMC hypersurfaces are invariant under a different $\K$-action (with respect to which the metric is assumed invariant). This is the right $\K$-action on a tube around $S$ described in Lemma~\ref{lemma:normalityB}. The kernel of this action is $\Hs$, so effectively this is a free $\K/\Hs$-action.
The unifying feature of these $2$ cases is that bifurcating branches issue from a natural $1$-parameter family of (homogeneous) CMC hypersurfaces that condense on a minimal submanifold $S$ and partially preserve the symmetries of the natural branch. The type of symmetry preservation is different according to $\Hs\triangleleft\K$ or $\K\triangleleft\G$, and we denote by $\GG$ the corresponding symmetry group, respectively $\K/\Hs$ or $\K$. In both cases, the geometric interpretation of these symmetries is that the Delaunay-type hypersurfaces are tubular graphs of (nonconstant) real-valued functions on $S$. As explained above, they accumulate on geodesic tubes, which are the tubular graphs of constant functions on $S$.

The main reason to call the above \emph{Delaunay-type hypersurfaces} is that both situations occur simultaneously in the classical case of Delaunay surfaces in $S^3$.
In this case, $\G=\mathsf T^2$ is a torus, $\K=\mathsf S^1$, and $\Hs$ is trivial, so both normality conditions hold. The corresponding Delaunay surfaces are CMC tori of unduloid type, that bifurcate from the family of Clifford tori, which condenses on a great circle $S$ in $S^3$, see Subsection~\ref{subsec:clifftori} for details.
The only cohomogeneity one $\G$-manifolds with a singular orbit such that both normality assumptions hold have trivial $\Hs$ and $\K=\mathsf S^1$ or $\K=\mathsf S^3$. This essentially forces the action to be a sum action on a sphere, as the above case of $S^3$, with generalized Clifford tori as principal orbits, see Subsections \ref{subsec:highdimclifftori} and ~\ref{subsec:kprod}. It is plausible to conjecture that a similar bifurcation result should hold without \emph{any} normality assumption (see Remark~\ref{rem:mcf}), although in this case no symmetry is to be expected for the CMC surfaces in the bifurcating branches.

The proof of the above result has two main components. First, establishing the appropriate cohomogeneity one setup; second, finding a corresponding bifurcation criterion with symmetry preservation. The cohomogeneity one framework is very natural in this context, since it provides a canonical $1$-parameter family of homogeneous CMC hypersurfaces from which Delaunay-type hypersurfaces can bifurcate. However, the spectral flow of the corresponding Jacobi operators is complicated, and can only be properly analized using the symmetry group $\GG$. At the same time, these are the symmetries to be preserved in the bifurcating branches, and finding them requires a deeper analysis of the geometry of cohomogeneity one manifolds, see Lemmas~\ref{lemma:normalityA} and \ref{lemma:normalityB}. Our bifurcation criterion with symmetry preservation uses these other symmetries and relies on conveniently combining a classical bifurcation criterion with the Symmetric Criticality Principle of Palais \cite{palais}, see Proposition~\ref{prop:symmpreserv}. Although we only state and use it in the CMC framework, in principle, this technique applies to other equivariant geometric variational problems and hence has some interest in its own right. For instance, similar techniques were recently used by the authors to study bifurcation of homogeneous solutions to the Yamabe problem~\cite{bp0,bp}.

Let us give a sample of Delaunay-type hypersurfaces $N$ invariant under a group $\GG$ that can be obtained in spheres and projective spaces as an application of the above Theorem. For simplicity, we refer to some of these submanifolds by their double coverings. Here, $T_1 S^k$ denotes the unit tangent bundle of the sphere $S^k$ and $\Sigma^{2k-1}$ denotes a possibly exotic sphere, see Sections~\ref{sec:ex1}, \ref{sec:ex2} and \ref{sec:ex3} for details.

\begin{table}[htf]
\begin{tabular}{|ccc|}
\hline
$M$ \rule{0pt}{2.5ex} \rule[-1.2ex]{0pt}{0pt} & $N$ & $\GG$ \\
\hline \hline
$S^{k+1}$  \rule{0pt}{2.5ex} & $S^1\times S^{k-1}$ &  $\SO(k)$ \\
$S^{2k+3}$ & $S^1\times S^{k+1}$ & $\mathsf S^1$\\
$\Sigma^{2k-1}$ & $S^1\times T_1 S^{k-1}$ & $\mathsf S^1$ \\
\hline
\end{tabular} \quad
\begin{tabular}{|ccc|}
\hline
$M$ \rule{0pt}{2.5ex} \rule[-1.2ex]{0pt}{0pt} & $N$ & $\GG$ \\
\hline \hline
$\C P^{k}$ \rule{0pt}{2.5ex} & $S^{2k-1}$ & $\mathsf S^1$ \\
$\Hr P^{k}$ & $S^{4k-1}$ & $\mathsf S^3$ \\
$\C P^{k}$ & $T_1 S^{k}$ & $\mathsf S^1$ \\
\hline
\end{tabular}
\end{table}
The Delaunay-type tori $S^1\times S^{k-1}$ in $S^{k+1}$ bifurcate from the family of Clifford tori, as previously observed in \cite{ap}. The Delaunay-type spheres in $\C P^k$ and $\Hr P^k$ bifurcate from a family of distance spheres around any given point, which are metrically Berger spheres, see Example~\ref{ex:cpn}. These and all other CMC hypersurfaces listed above are, to the best of our knowledge, new examples of non-homogeneous CMC hypersurfaces in such ambient spaces. In dimensions up to $7$ we use a classification of Hoelscher~\cite{hoelscher} to list all possible Delaunay-type hypersurfaces originating from \emph{primitive} cohomogeneity one actions, see Section~\ref{sec:ex4}. We stress that the applications of our main result have a much larger scope than the above concrete examples. In particular, we describe two procedures to generate other examples (of arbitrarily large dimension) out of known ones, see Section~\ref{sec:ex5}.

The paper is organized as follows. In Section~\ref{sec:cohom1} we recall basic geometric and topological aspects of cohomogeneity one manifolds and study normality conditions among isotropy groups. We also describe a convenient framework for invariant metrics and the spectrum of the Laplacian of principal orbits. Details about the variational formulation of the CMC problem and our main bifurcation criterion (Proposition~\ref{prop:symmpreserv}) are presented in Section~\ref{sec:variational}. In Section~\ref{sec:mainpf}, we prove our main result (Theorem~\ref{thm:main}), that immediately implies the Theorem stated above. Finally, Sections~\ref{sec:ex1}, \ref{sec:ex2}, \ref{sec:ex3}, \ref{sec:ex4} and \ref{sec:ex5} describe various examples and constructions to which the above results apply, providing many examples of Delaunay-type hypersurfaces.

\smallskip
\noindent
{\bf Acknowledgement.} It is a pleasure to thank Karsten Grove and Wolfgang Ziller for many valuable comments and suggestions.

\section{Geometry of cohomogeneity one manifolds}
\label{sec:cohom1}

In this section we lay down the basic framework of cohomogeneity one manifolds, fixing notation and describing the necessary hypotheses. Although the structure of such manifolds is well documented in the literature, see \cite{book,aa,gzricci,gzannals,gwz,vz}, we recall some of its basic aspects as a service to the reader.

\subsection{Topological structure}\label{subsec:topology}
A connected Riemannian manifold $(M,\g)$ is said to have \emph{cohomogeneity one} if it supports an isometric action of a compact Lie group $\G$ with at least one orbit of codimension $1$. This means that the orbit space $M/\G$ is one-dimensional, and since $\G$ is closed in the full isometry group of $(M,\g)$, we must have, up to renormalization, one of:
\begin{itemize}
\item[(i)] $M/\G=\R$, the real line;
\item[(ii)] $M/\G=[0,+\infty[$, a half-line;
\item[(iii)] $M/\G=S^1$, the circle;
\item[(iv)] $M/\G=[-1,1]$, an interval.
\end{itemize}
As mentioned above, interior points of $M/\G$ correspond to principal orbits and boundary points correspond to nonprincipal orbits, that can be singular or exceptional. Since we are interested in manifolds that have singular orbits, we leave aside cases (i) and (iii), where $M$ is a bundle over $M/\G$ with fiber a principal orbit. In case (ii), $M$ must be noncompact and there is exactly one nonprincipal orbit; however under the assumptions that $M$ is primitive and effective, this nonprincipal orbit is a point, see \cite[Thm 8.1 (iii)]{aa}. We also leave this case aside, since we are interested in singular orbits of positive dimension. Results analogous to ours in the case of singular orbits that are fixed points were obtained by Ye~\cite{ye1,ye2}. We are thus left with case (iv), which is the one that allows the most interesting topological constructions. Henceforth, we only deal with manifolds in this class.

Denote by $\pi\colon M\to M/\G=[-1,1]$ the projection map, and let $\gamma\colon\R\to M$ be a unit speed geodesic starting at a point $\gamma(0)\in M$ in a principal orbit, such that $\pi(\gamma(0))=0\in [-1,1]$, with initial velocity $\dot\gamma(0)$ perpendicular to the principal orbit $\pi^{-1}(0)$. Then $\gamma(t)$ meets all orbits perpendicularly, and is hence a horizontal geodesic.
The image $\gamma(\R)$ is either an embedded circle or a one-to-one immersed line \cite{aa}. Denote by $\Hs:=\G_{\gamma(0)}$ the isotropy group at $\gamma(0)$, which is equal to the isotropy groups $\G_{\gamma(t)}$ for all $t\neq 1\mod 2\Z$; and by $\K_\pm:=\G_{\gamma(\pm1)}$ the isotropy groups at $\gamma(\pm1)$, respectively. The group $\Hs$ is called \emph{principal isotropy group} and $\K_\pm$ are called \emph{singular} or \emph{exceptional isotropy groups}, according to the nonprincipal orbit $S_\pm:=\G/\K_\pm$ having dimension less than or equal to $\G/\Hs$, respectively. These are the possible \emph{orbit types} on a cohomogeneity one manifold. In this context, requiring the presence of a singular orbit $S$ that is also not a fixed point is a quite natural assumption. For instance, if $M$ is simply-connected and $\G$ is connected, there are no exceptional orbits \cite[Lemma 1.6]{gwz}, and hence any nonprincipal orbit must be singular.

Since $M/\G=[-1,1]$, the manifold $M$ is the union of tubular neighborhoods of $S_\pm$ glued along their common boundary, i.e., $M=D(S_-)\cup_{\pi^{-1}(0)} D(S_+)$, where $D(S_-):=\pi^{-1}([-1,0])$ and $D(S_+):=\pi^{-1}([0,1])$. Let $D_\pm$ be the \emph{slice} of $S_\pm=\G/\K_\pm$ at $\gamma(\pm1)$, i.e., the normal disk to the orbit $S_\pm$. Recall that $\K_\pm$ acts on $D_\pm$ via the slice representation. From the Slice Theorem, $D(S_\pm)$ is equivariantly diffeomorphic to $\G\times_{\K_\pm} D_\pm$, the associated bundle with fiber $D_\pm$ to the $\K_\pm$-principal bundle $\K_\pm\to \G\to \G/\K_\pm$. Thus, $M$ decomposes (up to equivariant diffeomorphism) as
\begin{equation}\label{eq:decomp}
M= \big(\G\times_{\K_-} D_-\big) \cup_{\G/\Hs} \big(\G\times_{\K_+} D_+\big).
\end{equation}
Moreover, $\K_\pm/\Hs=\partial D_\pm$ are spheres; more precisely, normal spheres to $S_\pm$ at $\gamma(\pm1)$, on which $\K_\pm$ acts transitively. In this way, $M$ is completely determined by its \emph{group diagram}
\begin{equation}\label{eq:groupdiagram}
\begin{gathered}
 \xymatrix@=10pt{& \G & \\ \K_- \ar[ur] & & \K_+ \ar[ul] \\ & \Hs\ar[ur]\ar[ul] &}
\end{gathered}
\end{equation}
where the arrows denote the natural inclusions. Conversely, given compact Lie groups $\Hs\subset\{\K_-,\K_+\}\subset \G$ as in the diagram above, such that $\K_\pm/\Hs$ are spheres, there is a unique (up to diagram equivalence) cohomogeneity one $\G$-manifold $M$ with principal isotropy $\Hs$ and singular isotropies $\K_\pm$, defined by \eqref{eq:decomp}, see~\cite{book,gwz} for details.
Finally, by the above group diagram, we have the homogeneous fibrations
\begin{equation}\label{eq:homfib}
\K_\pm/\Hs\longrightarrow \G/\Hs\stackrel{q_\pm}{\longrightarrow} \G/\K_\pm, \quad q_\pm(g\Hs):=g\K_\pm.
\end{equation}
The fibers of \eqref{eq:homfib} are the above mentioned normal spheres to $S_\pm=\G/\K_\pm$. Viewed inside $\G/\Hs$, they are the subsets formed by cosets of the form $gk\Hs$, i.e.,
\begin{equation}\label{eq:fibers}
q_\pm^{-1}(g\K_\pm) = (g\K_\pm)\Hs\subset \G/\Hs.
\end{equation}

\subsection{Normality assumptions}
Our main result requires one of the normality assumptions $\K\triangleleft\G$ or $\Hs\triangleleft\K$.\footnote{When we omit the subscripts $_\pm$ we are referring to either one of the two ``halfs" \eqref{eq:decomp} of $M$.}
Although both are strong requirements, the class of cohomogeneity one manifolds satisfying either of them is still very large. We now discuss a few issues related to each of these assumptions.

\subsubsection{$\K\triangleleft\G$}
The presence of a proper normal subgroup in a compact connected Lie group $\G$ is a strong condition. In particular, it implies that there exists a connected normal subgroup $\mathsf L\triangleleft\G$ such that $\G=\K\cdot\mathsf L$, i.e., $\G$ is a quotient of the product $\K\times\mathsf L$ by a finite central subgroup.

This normality condition also has strong geometric consequences in the context of a cohomogeneity one manifold. As introduced by Grove and Searle~\cite{gs}, an abstract isometric $G$-action on $M$ is called \emph{fixed point homogeneous} if the codimension of the fixed point set $M^G$ viewed inside the orbit space $M/G$ is equal to $1$. Equivalently, $G$ acts transitively on a normal sphere to some component of $M^G$.

\begin{lemma}\label{lemma:normalityA}
Let $M$ be a manifold with a cohomogeneity one $\G$-action having a normal singular isotropy $\K\triangleleft \G$. Then the $\K$-action on $M$ obtained as restriction of the $\G$-action is fixed point homogeneous and its principal orbits are the fibers $g\K\Hs$ of the corresponding homogeneous fibration \eqref{eq:homfib}.
\end{lemma}

\begin{proof}
Since $\K\triangleleft\G$, the left translation action of $\K$ on $\G/\K$ is trivial. Thus, the restriction to $\K$ of the $\G$-action on $M$ fixes the singular $\G$-orbit $S=\G/\K$. Moreover, $\K$ acts transitively on the normal sphere $\K/\Hs$ to $S$, hence the $\K$-action on $M$ is fixed point homogeneous, see \cite{gs}. Clearly, its principal orbits are $\K g\Hs=g\K\Hs$, i.e., the fibers \eqref{eq:fibers} of the homogeneous fibration $q\colon \G/\Hs\to\G/\K$.
\end{proof}

\begin{remark}
The subaction of $\K$ on $M$ is always well-defined, however its orbits may fail to be the fibers of \eqref{eq:homfib} if $\K$ is not normal in $\G$. For instance, consider the cohomogeneity one action of $\Spin(9)$ on the Cayley plane $\Ca P^2$. This action has a fixed point $S_-=\{p\}$, corresponding to $\K_-=\Spin(9)$, and the principal orbits are distance spheres $S^{15}$ centered at $p$, corresponding to $\Hs=\Spin(7)$. These principal orbits fiber over the other singular orbit, which is the cut locus of $p$, namely $S_+=S^8(1/2)$, corresponding to $\K_+=\Spin(8)$. The homogeneous fibration $\K_+/\Hs\to \G/\Hs\to \G/\K$ is the Hopf fibration $S^7\to S^{15}\to S^8(1/2)$. Clearly, $\K_+$ is not normal in $\G$, and we claim that the orbits of the $\K_+$-action on $\G/\Hs$ are not the fibers $S^7$ of the Hopf fibration. Instead, this $\K_+$-action on $S^{15}$ has itself cohomogeneity one, with singular isotropies $\Spin(7)$ and principal isotropy $\mathsf G_2$. Thus, the singular $\K_+$-orbits on $\G/\Hs$ are $S^7$ (and these are actually two antipodal Hopf fibers), but all the remaining orbits are hypersurfaces $\Spin(8)/\mathsf G_2$ of $S^{15}$.
\end{remark}

\subsubsection{$\Hs\triangleleft\K$}
We start by characterizing this algebraic condition geometrically, in terms of the slice representation at the singular orbit $\G/\K$.

\begin{proposition}\label{prop:hactstrivial}
Consider a cohomogeneity one manifold with principal isotropy $\Hs$ and singular isotropy $\K$. Then $\Hs\triangleleft\K$ if and only if $\Hs$ acts trivially on each slice $D$ of the singular orbit $S=\G/\K$ where $S$ intersects the fixed horizontal geodesic $\gamma$.
\end{proposition}

\begin{proof}
Assume $\Hs\triangleleft\K$, and reparametrize $\gamma$ so that $\gamma(0)=p\in S$. As mentioned above, $\K$ acts transitively on the normal sphere $\partial D=\K/\Hs$ to $S$ with isotropy $\Hs$, so the $\K$-orbit of $\dot\gamma(0)$ contains all normal directions to $S$, i.e., all directions tangent to $D$. Thus, the isotropy of every point in $D\setminus\{p\}$ is a conjugate of $\Hs$ in $\K$, which hence must be equal to $\Hs$; and $\G_p=\K$ also contains $\Hs$. Hence $\Hs$ acts trivially on $D$. The converse statement follows similarly.
\end{proof}

An immediate consequence of $\Hs\triangleleft\K$ is that the quotient $\K/\Hs$ has a group structure. At the same time, $\K/\Hs$ is a sphere (the normal sphere to $\G/\K$ inside $M$). Therefore, $\K/\Hs$ must be diffeomorphic to $S^0=\Z_2$, $S^1$ or $S^3$, which are the only spheres to admit a group structure. The case $\K/\Hs=\Z_2$ can only occur when $\G/\K$ is an exceptional orbit (i.e., $\dim \G/\K=\dim \G/\Hs$) and is hence already excluded by another hypothesis. Thus, the codimension of the singular orbit $S=\G/\K$ must be equal to $2$ or $4$ under the assumption $\Hs\triangleleft\K$. Let us discuss a few converse statements, that can be used to check if $\Hs\triangleleft\K$ when $\K/\Hs=S^1$.

\begin{proposition}\label{prop:HnormalinK1}
Let $\K$ be a Lie group and let $\Hs\subset \K$ be a compact subgroup such that $\K/\Hs$ is homeomorphic to $S^1$. If there exists a subgroup $\mathsf N\triangleleft\K$, $\mathsf N\subset \Hs$, that intercepts all connected components of $\K$, then $\Hs\triangleleft\K$.
\end{proposition}

\begin{proof}
Replacing $\mathsf N$ by its closure, we can assume $\mathsf N$ closed. Then, $(\K/\mathsf N)/(\Hs/\mathsf N)\cong \K/\Hs\cong S^1$. Since $\Hs/\mathsf N$ is compact, the circle admits a $\K/\mathsf N$-invariant metric, which must be the round metric of some radius. Thus, the $\K/\mathsf N$-action is given by a homomorphism $\phi\colon\K/\mathsf N\to\mathsf O(2)$ whose image is contained in $\SO(2)$, since $\K/\mathsf N$ is connected. Thus, the action is free and $\ker\phi=\Hs/\mathsf N\triangleleft\K/\mathsf N$, hence $\Hs\triangleleft\K$.
\end{proof}

\begin{remark}
If $\K$ is connected and $\K/\Hs\cong S^1$, then setting $\mathsf N=\{1\}$ we automatically have $\Hs\triangleleft\K$.
Note that without the assumption on $\mathsf N$, the result in general fails. For instance, take $\K=\mathsf O(2)$ and $\Hs\cong\mathds Z_2\subset\mathsf O(2)$ the subgroup generated by some reflection. Then $\K/\Hs\cong S^1$, but $\Hs$ is not normal in $\K$.
\end{remark}

\begin{remark}
The above cannot be adapted to $\K/\Hs=S^3$. For instance, note that $\mathsf{SO}(4)/\mathsf{SO}(3)=\U(2)/\U(1)=S^3$, but $\mathsf{SO}(3)\ntriangleleft\mathsf{SO}(4)$ and $\U(1)\ntriangleleft\U(2)$.
\end{remark}

A different approach to the connectedness of a singular isotropy is by topological assumptions on the cohomogeneity one $\G$-manifold and the group $\G$, as follows.

\begin{corollary}\label{cor:normality1conn}
Let $M$ be a simply-connected manifold with a cohomogeneity one action by a connected Lie group $\G$, with $M/\G=[-1,1]$. Denote the nonprincipal orbits by $S_\pm=\G/\K_\pm$ and suppose $\operatorname{codim} S_-=2$ and $\operatorname{codim} S_+\geq 3$. Then $\Hs\triangleleft\K_-$.
\end{corollary}

\begin{proof}
From \cite[Lemma 1.6]{gwz}, $\K_-$ is connected, even more, $\K_-=\Hs_0\cdot S^1$ and $\Hs=\Hs_0\cdot\Z_k$, where $\Hs_0$ denotes the identity component of $\Hs$. Since $S_-=\K_-/\Hs$ is homeomorphic to $S^1$, Proposition~\ref{prop:HnormalinK1} applies.
\end{proof}

\begin{remark}\label{rem:so3s4}
The above result is optimal, in the sense that it fails if $\operatorname{codim} S_- =\operatorname{codim} S_+=2$. For example, consider the isometric $\mathsf{SO}(3)$-action on the round sphere $S^4$ with group diagram
$\Z_2\oplus\Z_2 \subset \{\mathsf{S(O(2)O(1))},\mathsf{S(O(1)O(2))}\}\subset \mathsf{SO}(3),$
where the embedding of $\Hs=\Z_2\oplus\Z_2$ in $\K_\pm$ is as $\mathsf{S(O(1)O(1)O(1))}$, see \cite[Table F]{gwz}. Both singular orbits $S_\pm$ are images of Veronese embeddings of $\R P^2$ in $S^4$, and have codimension $2$, but $\Hs$ is not normal in $\K_-$ or in $\K_+$.
\end{remark}

Finally, under the normality condition $\Hs\triangleleft\K$ we also have a geometric interpretation of the fibers $g\K\Hs$ of the homogeneous fibration $q\colon \G/\Hs\to \G/\K$ in terms of group orbits. Since $\Hs\triangleleft\K$, the right $\K$-action on $\G/\Hs$ given by
\begin{equation}\label{eq:krightaction}
\G/\Hs\times\K\ni (g\Hs,k)\longmapsto g\Hs k=gk\Hs \in \G/\Hs
\end{equation}
is well-defined. Moreover, the normal subgroup $\Hs$ of $\K$ is precisely the kernel of this action. Hence, taking the quotient, we get a free right action of the group $\K/\Hs$ on $\G/\Hs$, with the same orbits as \eqref{eq:krightaction}. We now extend this action on each principal orbit to an action on a tubular neighborhood of the singular orbit $S=\G/\K$.

\begin{lemma}\label{lemma:normalityB}
Let $M$ be a manifold with a cohomogeneity one $\G$-action with a singular orbit $S=\G/\K$ and $\Hs\triangleleft\K$. There is a smooth right action of $\K/\Hs$ on the tubular neighborhood $D(S)$, which fixes every point in $S$ and its other orbits are the fibers $g\K\Hs$ of the corresponding homogeneous fibration \eqref{eq:homfib}.
\end{lemma}

\begin{proof}
Let $p\in S$ be a point where $S$ intersects the fixed horizontal geodesic $\gamma$. For convenience, reparametrize $\gamma$ so that $\gamma(0)=p$. Since the $\K$-action on the normal sphere to $S$ is transitive, the set
$D=\big\{k\cdot \gamma(t):k\in\K, |t|<\varepsilon\big\}$
is a slice of $S$ at $p$. Consider the right $\K$-action on $D$ given by
\begin{equation}\label{eq:kactiond}
 \big(k\cdot \gamma(t)\big)\cdot\overline k=k\overline k\cdot \gamma(t), \quad \mbox{ for all }  |t|<\varepsilon.
\end{equation}
By Proposition~\ref{prop:hactstrivial}, we have $k\cdot\gamma(t)=kh\cdot\gamma(t)$ for all $h\in \Hs$. Thus, since $\Hs\triangleleft\K$, the right action \eqref{eq:kactiond} is well-defined and smooth. Effectively, this is a right action of the quotient group $\K/\Hs$ on $D$. The corresponding orbits are clearly normal spheres to $S$ that foliate $D$, and $p\in S\cap D$ is a fixed point. Letting $\G$ act on the slice $D$ we immediately get an extension of this right $\K/\Hs$-action to the tubular neighborhood $D(S)=\G\cdot D$ with the desired properties.
\end{proof}

\begin{remark}
Let $M$ be a cohomogeneity one manifold with group diagram $\Hs\subset\{\K_-,\K_+\}\subset \G$, where $\Hs\triangleleft\K_-$. We note that the above $\K_-/\Hs$-action on the tubular neighborhood $D(S_-)$ of $S_-=\G/\K_-$ might not extend smoothly to the entire manifold $M$. It is clear that $\K_-/\Hs$ acts smoothly on $M\setminus S_+$, but the action extends smoothly across $S_+$ if and only if $\Hs$ is also normal in $\K_+$. Note that, in this case, the $\G$-action on $M$ cannot be primitive.
\end{remark}

\subsection{Collapse of principal orbits}\label{subsec:collapse}
Consider the family of principal orbits
\begin{equation}\label{eq:porbits}
x_t\colon \G/\Hs\hookrightarrow M, \quad x_t(\G/\Hs)=\G\cdot\gamma(t), \quad t\in \,]-1,1[.
\end{equation}
In order to describe the geometry of this family near a singular orbit, i.e., as $t$ goes to the boundary of $M/\G=[-1,1]$, it suffices to focus on one of the tubular neighborhoods $D(S_\pm)$ in  \eqref{eq:decomp}, or ``halfs", at a time. Denote this tubular neighborhood by $D(S)\cong\G\times_\K D$, the corresponding singular orbit by $S=\G/\K$, and reparameterize the orbit space so that the singular orbit $S$ corresponds to $t=0$, while nearby principal orbits correspond to $t>0$; i.e., $(\G\times_\K D)/G=[0,\varepsilon)$.

As previously mentioned, the homogeneous hypersurface $x_t(\G/\Hs)$ is a geodesic tube around $S$, i.e., consists of points at fixed distance\footnote{The distance between any point at $x_t(\G/\Hs)$ and $S$ is $t$ provided that the horizontal geodesic $\gamma(t)$ is parametrized by arc length.} $t$ from $S$. In other words, $x_t(\G/\Hs)$ is the total space of a sphere bundle
\begin{equation}\label{eq:sphbundle}
\K/\Hs\longrightarrow x_t(\G/\Hs)\stackrel{q}{\longrightarrow} S=\G/\K,
\end{equation}
where the fiber $\K/\Hs=\partial D$ is the normal sphere to $S$, and $q$ is the homogeneous fibration \eqref{eq:homfib}, modulo using $x_t$ as an identification. Because of the $\G$-equivariance of its mean curvature vector, each orbit $x_t(\G/\Hs)$ is automatically CMC. This family of CMC homogeneous hypersurfaces $x_t(\G/\Hs)$, $t>0$, foliates the tubular neighborhood $\G\times_\K D$ of $S$. As $t\to 0$, the hypersurfaces $x_t(\G/\Hs)$ get arbitrarily close to $S$ in the Hausdorff metric, as the normal sphere that is the fiber of \eqref{eq:sphbundle} collapses to a point. In this case, we say $x_t(\G/\Hs)$ \emph{collapses} to $S$, or \emph{condenses} on $S$. Moreover, since $S$ is an isolated singular orbit it is automatically a minimal submanifold \cite{book,hsiang,hl}. This means we are precisely in the situation discussed in \cite{mmp,mp}, where CMC hypersurfaces collapse, or condense, on a minimal submanifold $S$.

\begin{remark}
The above property that the \emph{limit} submanifold on which CMC hypersurfaces condense is minimal may fail to hold in general. Such property holds when the asymptotic behaviors of the mean curvature of these hypersurfaces and of the norm of their shape operator are comparable. In this way, the limit submanifolds can be proven to be minimal in cases much more general than cohomogeneity one, but using rather elaborate geometric measure theory arguments \cite[Thm 6.1]{mp}.
\end{remark}

It is not hard to see that the closer a CMC hypersurface is to a lower dimensional minimal submanifold (in the Hausdorff metric), the larger its mean curvature must be. In particular, the mean curvature $\H_t$ of $x_t(\G/\Hs)$ diverges to $+\infty$ as $t\to0$. Moreover, also the first derivative $\H_t'\to+\infty$ as $t\to0$. Since $\H_t$ is the trace of the shape operator $\mathcal S_t:=\mathcal S_{x_t}$, its Hilbert-Schmidt norm also blows up. In short,
\begin{equation}\label{eq:sxtblowsup}
\H_t \to+\infty,\; \H_t'\to+\infty\text{ and } \big\|\mathcal S_{t}\big\|\to+\infty, \; \text{ as } x_t(\G/\Hs) \text{ collapses to } \G/\K.
\end{equation}

Let us now go back to the \emph{global} setup of a cohomogeneity one manifold with $M/\G=[-1,1]$, where $t=\pm1$ correspond to the two singular orbits. In addition to those singular orbits, which are minimal submanifolds, there is also at least one minimal principal orbit. This follows from the fact that a $\G$-orbit of extremal volume among orbits of the same type is minimal, see \cite{book,hsiang,hl}. Thus, the mean curvature function $\H_t$ of \eqref{eq:porbits} blows up as $t\to\pm1$ but also attains the global minimum $\H_{t_0}=0$ for some $t_0\in\,]-1,1[$ which corresponds to an orbit $x_{t_0}(G/H)$ of maximal volume among principal orbits. This minimal principal orbit need not be unique.

\begin{remark}
By Frankel's Theorem, any two minimal hypersurfaces in an ambient manifold with $\Ric>0$ must intersect, hence $\Ric>0$ is a sufficient condition for uniqueness of the minimal principal orbit. A modified version \cite{petwil} of this result, or more generally the warped product splitting theorem
\cite{ck} gives the interesting consequence that if $\Ric\geq0$, then any two distinct minimal principal orbits $N_1$ and $N_2$ must be totally geodesic and the region between these hypersurfaces splits as a product $N_1\times [a,b]$, with boundary at $N_1\cong N_1\times \{a\}$ and $N_2\cong N_1\times \{b\}$. This means that, in this case, $\H_t=0$ for $t$ in an interval of positive measure in $]-1,1[$. As a side note, we recall that a cohomogeneity one manifold $M$ always admits an invariant metric with $\Ric\geq0$, and it admits an invariant metric with $\Ric>0$ if and only if $\pi_1(M)$ is finite \cite{gzricci}.
\end{remark}

\begin{remark}\label{rem:mcf}
The above statements about $\H_t$ correspond to statements regarding the dynamics of the \emph{Mean Curvature Flow} (MCF) of orbits on $M$, i.e., the $L^2$-gradient flow of the area functional. The MCF preserves the symmetries of $M$, i.e., an orbit of an isometric action evolves under the MCF through other orbits of the same type, as observed in \cite[Thm 6.1]{lt} and \cite[Thm 2]{pacini}, see also \cite{alex-rad}. In this way, the MCF of orbits on $M$ may be seen as a flow on $M/\G$ that preserves the orbit type stratification.
The limit of a nonminimal principal orbit evolving under the MCF is always a minimal orbit, which is singular if the MCF ends in finite time, and regular otherwise.
It is natural to expect that the bifurcation results in this paper generalize to bifurcation of CMC hypersurfaces issuing from paths of principal orbits that are solutions to the MCF collapsing to a minimal submanifold in finite time.
Although this approach is suggestive of a more general construction of Delaunay-type hypersurfaces, it is not clear whether a suitable counterpart of the partial symmetry preservation phenomenon would be available in this framework. We stress that this symmetry property is a crucial feature of our bifurcating solutions, that justifies calling such objects Delaunay-type hypersurfaces.
\end{remark}

\subsection{Adapted metrics}
A cohomogeneity one $\G$-manifold $M$ admits many Riemannian metrics $\g$ for which the $\G$-action is isometric. Any such $\G$-invariant metric $\g$ on $M$ is determined by its restriction to the (open and dense) subset $M_0=M\setminus S_\pm$ consisting of principal orbits. More precisely, using the fixed horizontal geodesic $\gamma(t)$, we can describe $\g$ as a one-parameter family of $\G$-invariant metrics $\g_t$ on $\G/\Hs$; i.e., using the equivariant diffeomorphism $M_0\cong \left]-1,1\right[\times \G/\Hs$,
\begin{equation}\label{eq:ggt}
\g=\dd t^2+\g_t, \quad -1<t<1,
\end{equation}
where $\g_t$ satisfies suitable smoothness conditions at $\pm1$. Note that $\g_t$ is the metric induced by $\g$ on the image of \eqref{eq:porbits}, i.e., $x_t\colon (\G/\Hs,\g_t)\hookrightarrow (M,\g)$ is an isometric embedding for each $-1<t<1$.
Basic material on cohomogeneity one metrics can be found in \cite{book,gzricci,gzannals}, and a method to determine the smoothness conditions at $\pm1$ necessary for \eqref{eq:ggt} to extend to a smooth metric on $M$ is given in \cite{vz}. 

Once more, we are interested in a property of a cohomogeneity one metric near one of the singular orbits $t=\pm1$. For this reason, as in Subsection \ref{subsec:collapse}, consider a tubular neighborhood $\G\times_\K D$ of a singular orbit $S=\G/\K$, and suppose $S$ corresponds to $t=0$ in the orbit space and nearby principal orbits $\G/\Hs$ to $t>0$. Denote by $\mathfrak h\subset\mathfrak k\subset\mathfrak g$ the Lie algebras of $\Hs\subset \K\subset\G$, respectively. We henceforth fix an $\Ad(\K)$-invariant complement $\mathfrak m$ to $\mathfrak k$ in $\mathfrak g$, and an $\Ad(\Hs)$-invariant complement $\mathfrak p$ to $\mathfrak h$ in $\mathfrak k$, i.e.,
\begin{equation*}
\mathfrak k\oplus\mathfrak m=\mathfrak g, \quad [\mathfrak k,\mathfrak m]\subset\mathfrak m \quad \mbox{ and } \quad \mathfrak h\oplus\mathfrak p=\mathfrak k, \quad [\mathfrak h,\mathfrak p]\subset\mathfrak p.
\end{equation*}
For example, $\mathfrak m$ and $\mathfrak p$ can be defined as the orthogonal complements of $\mathfrak k$ in $\mathfrak g$ and of $\mathfrak h$ in $\mathfrak k$ with respect to an auxiliary bi-invariant metric. There are natural identifications (via action fields) of $\mathfrak m$ and $\mathfrak p$ with tangent spaces to $\G/\K$ and $\K/\Hs$. Moreover, the sum
\begin{equation}\label{eq:n}
\mathfrak n:=\mathfrak m\oplus\mathfrak p
\end{equation}
is an $\Ad(\Hs)$-invariant complement to $\mathfrak h$ in $\mathfrak g$, which is identified with a tangent space to $\G/\Hs$. Recall that the set of $G$-invariant metrics on an abstract homogeneous space $G/H$ is naturally isomorphic to the set of $\Ad(H)$-invariant inner products on a given $\Ad(H)$-invariant complement of the Lie algebra of $H$, see e.g.~\cite{book}.
Thus, given a cohomogeneity one metric \eqref{eq:ggt}, the corresponding family $\g_t$ of metrics on $\G/\Hs$ can be identified with a family of $\Ad(\Hs)$-invariant inner products $\langle\cdot,\cdot\rangle_t$ on $\mathfrak n$.

\begin{definition}\label{def:adapted}
The cohomogeneity one metric metric \eqref{eq:ggt} is said to be \emph{adapted near $S$} if there exists a positive smooth function $\alpha\colon]0,\varepsilon[\,\to\R$ with $\lim_{t\to0}\alpha(t)=1$, such that the inner products $\langle\cdot,\cdot\rangle_t$ on $\mathfrak n$ that determine $\g_t$ are of the form
\begin{equation}\label{eq:adapted}
\langle\cdot,\cdot\rangle_t=\alpha^2(t)\,A(\cdot,\cdot) + B_t(\cdot,\cdot), \quad \mbox{ for } t \mbox{ close to } 0,
\end{equation}
where $A$ is the $\Ad(\K)$-invariant inner product on $\mathfrak m$ that induces the $\G$-invariant metric on $S=\G/\K$ inherited from $\g$ and $B_t$ is a $1$-parameter family of $\Ad(\Hs)$-invariant inner products on $\mathfrak p$. We also use the convention that if $S$ is a point, any smooth cohomogeneity one metric is adapted near $S$.
\end{definition}

The above condition on a cohomogeneity one metric is satisfied if and only if, near the singular orbit $S$, the corresponding $\G$-invariant metric $\g_t$ on $\G/\Hs$ is induced by an inner product on $\mathfrak n=\mathfrak m\oplus\mathfrak p$ that is not only $\Ad(\Hs)$-invariant on $\mathfrak m\oplus\mathfrak p$, but also turns the above into an orthogonal direct sum, is $\Ad(\K)$-invariant on $\mathfrak m$, and its component in $\mathfrak m$ is a homothety of the $\Ad(\K)$-invariant inner product that induces the invariant metric on $S$. For example, if $\mathfrak m$ and $\mathfrak p$ do not share equivalent $\Ad(\Hs)$-submodules, then all $\Ad(\Hs)$-invariant inner products $\langle\cdot,\cdot\rangle_t$ split as a sum of an $\Ad(\K)$-invariant inner product on $\mathfrak m$ and an $\Ad(\Hs)$-invariant inner product on $\mathfrak p$. In order to be adapted near $S$, one then needs that the metric component in $\mathfrak m$ is a homothety of the $\Ad(\K)$-invariant inner product that gives the metric on $S$. We stress that there are no obstructions to the existence of adapted metrics:

\begin{proposition}\label{prop:alwaysadapted}
Every cohomogeneity one $\G$-manifold $M$ with a singular orbit $S$ admits a metric that is adapted near $S$.
\end{proposition}

\begin{proof}
A slightly stronger statement holds, namely, we can take the corresponding functions $\alpha$ to be constants. Let $Q$ be a bi-invariant metric on $\G$ and set, for $t$ near $0$, $\alpha=1$, $A:=Q|_{\mathfrak m}$ and $B_t:=\sum_j (\beta_j(t))^2 Q|_{\mathfrak p_j}$, where $\mathfrak p=\bigoplus_j \,\mathfrak p_j$ is a sum of mutually $Q$-orthogonal $\Ad(\Hs)$-invariant subspaces and $\beta_j$ are positive smooth functions that vanish at $0$. We claim that the corresponding family \eqref{eq:adapted} gives a smooth invariant metric $\g=\dd t^2+\g_t$ on $M$. This follows from verifying that all required smoothness conditions on $\alpha$ and $\beta_j$ at $t=0$ can be satisfied if $\alpha$ are chosen constant, see \cite{vz} or \cite[Sec 2]{gzricci} for details.
\end{proof}

As a direct consequence of the proof of the Soul Conjecture by Perelman~\cite{perelman}, we get a sufficient condition for an invariant metric to be adapted near both of its singular orbits.

\begin{proposition}\label{prop:soul}
Let $M$ be a cohomogeneity one manifold with an invariant metric $\g$ of nonnegative sectional curvature. If $(M,\g)$ has a totally geodesic principal orbit, then the metric $\g$ is adapted near any of its singular orbits.
\end{proposition}

\begin{proof}
Let $S$ be a singular orbit. Again, a stronger result holds, in that the corresponding function $\alpha$ of $\g$ is constant equal to $1$. Denote by $N$ the totally geodesic principal orbit, and assume, up to taking double coverings, that $N$ disconnects $M$. The component $C$ of $M\setminus N$ containing $S$ is a locally convex subset of $M$, since the distance function to its boundary $N$ is concave. The subset of $C$ of points at maximal distance from $N$ is precisely $S$, called the \emph{soul} of $C$. Due to Perelman~\cite[Thm (C)]{perelman}, the Sharafutdinov retraction onto the soul is a Riemannian submersion. In our case, the restriction of this retraction to a principal orbit $x_t(\G/\Hs)$ is precisely the projection $q\colon x_t(\G/\Hs)\to S$, see \eqref{eq:sphbundle}. Consequently, $\alpha$ is constant equal to $1$ in \eqref{eq:adapted} and hence $\g$ is adapted near $S$.
\end{proof}

\begin{example}
If both singular orbits of a cohomogeneity one manifold $M$ have codimension $2$, then $M$ admits an invariant metric with nonnegative sectional curvature and a totally geodesic principal orbit, see Grove and Ziller \cite[Thm E]{gzannals}. In particular, these Grove-Ziller metrics are adapted near both singular orbits.
\end{example}

\subsubsection{$\K/\Hs$-invariance}
One of the hypotheses of the Theorem in the Introduction is that the metric is $\K$-invariant near $S$. In the case $\Hs\triangleleft\K$, this means that the right $\K/\Hs$-action on $D(S)$ given by Lemma~\ref{lemma:normalityB} is isometric, i.e., the metrics $\g_t$ on $\G/\Hs$ must be $\K/\Hs$-invariant with respect to \eqref{eq:krightaction} for $t$ near the boundary point corresponding to $S$. Such invariance of $\g_t$ can be algebraically described as follows.

\begin{lemma}
An inner product $\langle\cdot,\cdot\rangle_t$ on $\mathfrak n=\mathfrak m\oplus\mathfrak p$ gives rise to a $\K/\Hs$-invariant metric $\g_t$ on $\G/\Hs$ if and only if $\langle\cdot,\cdot\rangle_t|_{\mathfrak m}$ is $\Ad(\K)$-invariant, $\langle\cdot,\cdot\rangle_t|_{\mathfrak p}$ is bi-invariant and $\mathfrak m$ and $\mathfrak p$ are $\langle\cdot,\cdot\rangle_t$-orthogonal.
\end{lemma}

\begin{remark}
Analogously to Proposition~\ref{prop:alwaysadapted}, it follows by using a bi-invariant metric on $\G$ that any cohomogeneity one manifold satisfying $\Hs\triangleleft\K$ can be endowed with a metric adapted near $S=\G/\K$ and $\K/\Hs$-invariant on $D(S)$.
\end{remark}

\begin{remark}\label{rem:totgeod}
In the above situation where there is an isometric $\K/\Hs$-action on $D(S)$ fixing $S$, it follows that the singular orbit $S$ is not only minimal but also totally geodesic. This means that if $S$ is not totally geodesic in $(M,\g)$, then either $\Hs\ntriangleleft\K$ or $\g$ is not $\K/\Hs$-invariant near $S$. For instance, in the case of the $\SO(3)$-action on $S^4$ described in Remark~\ref{rem:so3s4}, neither of the singular orbits (Veronese embeddings of $\R P^2$) are totally geodesic in $S^4$ with the round metric, and $\Hs\ntriangleleft\K_\pm$.
\end{remark}

\subsection{Spectrum of principal orbits}
\label{sub:spectrumorbits}
In this subsection, assume $(M,\g)$ is a cohomogeneity one manifold with a metric adapted near a singular orbit $S=\G/\K$. We are interested in comparing the spectra of the Laplacian of $S$ to that of a nearby principal orbit $(\G/\Hs,\g_t)$. Since the metric is adapted near $S$, the homogeneous fibration $q\colon (\G/\Hs,\g_t)\to S$ given by \eqref{eq:sphbundle} is a Riemannian submersion when the metric on $S$ is rescaled by $\alpha^2$. Moreover, since $\g_t$ splits as a sum of an $\Ad(\K)$-invariant inner product on $\mathfrak m$ and an $\Ad(\Hs)$-invariant inner product on $\mathfrak p$, it follows that this Riemannian submersion has totally geodesic fibers, see \cite[Thm 9.80]{besse}.

Denote by $\Delta_t$, respectively $\Delta_{S}$, the (unique self-adjoint extension of the) Laplacian acting on the space of $L^2$ functions on the principal orbit $(\G/\Hs,\g_t)$, respectively on the singular orbit $S$. For each function $\psi\colon S\to\R$, we denote by $\widetilde\psi\colon \G/\Hs\to\R$ its \emph{lift}, i.e., $\widetilde\psi:=\psi\circ q$. Since \eqref{eq:sphbundle} has totally geodesic fibers,
\begin{equation}\label{eq:laplasubm}
\Delta_t\,\widetilde\psi = (\Delta_{\alpha^2S} \,\psi)\circ q=\tfrac{1}{\alpha^2}(\Delta_{S} \,\psi)\circ q.
\end{equation}
In particular, if $\psi$ is an eigenfunction of $\Delta_S$ with eigenvalue $\lambda$, then its lift $\widetilde\psi$ is an eigenfunction of $\Delta_t$ with eigenvalue $\lambda/\alpha^2$ (and constant along the fibers). Therefore, there is a natural inclusion
\begin{equation}\label{eq:inclusion}
\tfrac{1}{\alpha^2}\spec\big(\Delta_S\big)\subset\spec\big(\Delta_t\big).
\end{equation}
Conversely, if $\psi\colon \G/\Hs\to\R$ is constant along the fibers and satisfies $\Delta_t\,\psi=\lambda\psi$, then there exists $\check\psi\colon S\to\R$ such that $\psi=\check\psi\circ q$ and $\Delta_{\alpha^2S}\,\check\psi=\lambda\check\psi$. Summing up, it follows from \eqref{eq:laplasubm}, after verifying the adequate regularity hypotheses, that the following holds:

\begin{proposition}\label{prop:eigenlift}
An eigenfunction of $\Delta_t$ is constant along the fibers of \eqref{eq:sphbundle} if and only if it is the lift of an eigenfunction of $\Delta_S$.
\end{proposition}

Although the spectrum of $\Delta_t$ varies with $t$, certain eigenvalues remain of the form $\tfrac{\lambda}{\alpha^2}$, where $\lambda$ is a constant. These are precisely the eigenvalues that arise from the inclusion \eqref{eq:inclusion}. We call such eigenvalues \emph{basic eigenvalues} of $\Delta_t$.

\section{CMC Variational Formulation and Equivariant Bifurcation}\label{sec:variational}

Let $(M,\g)$ be a Riemannian manifold and $N\subset M$ a submanifold that bounds an open bounded subset $\Omega$ of $M$. It is a classical result that the embedding $x\colon N\hookrightarrow M$ has CMC if and only if it solves the following isoperimetric problem: the \emph{area} of $x(N)$ is critical among embeddings of $N$ into $M$, constrained to variations by embeddings whose image is the boundary of an open subset of $M$ with \emph{same volume} as $\Omega$. Equivalently, $x\colon N\hookrightarrow M$ has CMC if and only if it is a critical point of a Lagrange multiplier problem for area with a volume constraint. In this situation, the value of the Lagrange multiplier is the mean curvature.

\subsection{Variational formulation}
Consider the space $\Emb(N,M)$ of unparameterized\footnote{Since it is geometrically natural not to distinguish embeddings $N\hookrightarrow M$ obtained by right-composition by a diffeomorphism of $N$, we consider the quotient of the space of embeddings by the action of the diffeomorphism group of $N$. Elements of this quotient are called \emph{unparameterized} embeddings.} embeddings of class $C^{2,\alpha}$ of $N$ into $M$, and for $\H\in\R$, define
\begin{equation}\label{eq:functional}
\begin{aligned}
f_\H &\colon \Emb(N,M)\longrightarrow \R \\
f_\H(x)&=\mathrm{Area}(x)+\H\, \mathrm{Vol}(x),
\end{aligned}
\end{equation}
where $\mathrm{Area}(x)$ is the volume of the image $x(N)$ and $\mathrm{Vol}(x)$ is the volume of the open subset of $M$ whose boundary\footnote{In general, $x(N)$ might not be the boundary of an open subset, and under some extra hypotheses one can define a generalized notion of enclosed volume for such embeddings. Nevertheless, all embeddings considered in this paper are boundaries, hence we do not need to dwell on this issue.} is $x(N)$. Although $\Emb(N,M)$ does not have a natural global differentiable structure, around each \emph{smooth} embedding $x_0\colon N\hookrightarrow M$ it is easy to construct a local chart identifying nearby embeddings with sections of the normal bundle of $x_0$, via the exponential map of $M$. Since the codimension is one, assuming $x_0$ is transversely oriented (i.e., its normal bundle is oriented), this means we identify embeddings $x$ near $x_0$ with functions $\psi_x\in C^{2,\alpha}(N)$ in a neighborhood of the origin, so that (up to diffeomorphisms of $N$),
\begin{equation}
x(p)=\exp_{x_0(p)}\big(\psi_x(p) \cdot \vec n_{x_0} \big), \quad p\in N,
\end{equation}
where $\vec n_{x_0}$ is a unit normal vector field to $x_0(N)$. In this way, in a small chart around $x_0$, the functional $f_\H$ is smooth. As mentioned above, its critical points are embeddings (near $x_0$) with constant mean curvature equal to $\H$. For more details on the infinite-dimensional topological manifold structure of the space of unparameterized embeddings and regularity issues, see \cite{ap2}.

In the cohomogeneity one setup described in Section~\ref{sec:cohom1}, we have the family \eqref{eq:porbits} of  embeddings of principal orbits $x_t\colon \G/\Hs\hookrightarrow M$, whose mean curvature we denote $\H_t$, for each $t\in\,]-1,1[$. Notice that the images of $x_t$ are the common boundary of tubular neighborhoods $D(S_\pm)$ of the singular orbits $S_\pm$, see \eqref{eq:decomp}, so $\mathrm{Vol}(x_t)$ is the volume of one such tubular neighborhood, depending on the transverse orientation of $x_t(\G/\Hs)$. In this way, we encode the family \eqref{eq:porbits} as a $1$-parameter family of critical points $x_t\in\Emb(\G/\Hs,M)$ of the corresponding functionals $f_{\H_t}$, i.e.,
\begin{equation}\label{eq:critpath}
\dd f_{\H_t}(x_t)=0, \quad\mbox{ for all }t\in\,]-1,1[.
\end{equation}

\subsection{Second variation}
Under the above identifications, the second variation of $f_\H$ at a critical point $x$ is given by the following quadratic form on $C^{2,\alpha}(N)$
\begin{equation}\label{eq:secondvar}
\dd^2 f_\H(x)(\psi,\psi)=\int_N \psi\,\Delta_x \psi -\big(\Ric(\vec n_x)+\|\mathcal S_x\|^2\big)\psi^2,
\end{equation}
where $\Delta_x$ is the Laplacian of $(N,x^*\g)$, $\vec n_x$ is a unit normal vector field to $x(N)$ and $\|\mathcal S_x\|$ is the Hilbert-Schmidt norm of the shape operator of $x$. The above quadratic form is represented by the (formally) self-adjoint, linear elliptic differential operator
\begin{equation}\label{eq:jacobi}
J_x(\psi)=\Delta_x \psi -\big(\Ric(\vec n_x)+\|\mathcal S_x\|^2\big)\psi,
\end{equation}
called \emph{Jacobi operator}, or also \emph{stability operator}. A critical point $x$ is \emph{nondegenerate} (in the usual sense of Morse theory) if and only if $J_x$ has trivial kernel. Moreover, the \emph{Morse index} $i(x)$ of $x$ is the dimension of a maximal subspace where \eqref{eq:secondvar} is negative-definite, i.e., the number of positive eigenvalues of $\Delta_x$, counted with multiplicity, that are less than $\big(\Ric(\vec n_x)+\|\mathcal S_x\|^2\big)$.

\subsection{Bifurcation}
We now briefly discuss general bifurcation theory of general CMC embeddings, that will be applied to our cohomogeneity one setup. Let $(M,\g)$ be a Riemannian manifold and let
\begin{equation}\label{eq:xtabstract}
x_t\colon N\hookrightarrow M, \quad t\in [a,b],
\end{equation}
be a $C^1$ path of embeddings with CMC equal to $\H_t$.
Assume that $\H_t'\neq 0$ for all $t\in [a,b]$, cf. \eqref{eq:sxtblowsup} and \eqref{eq:critpath}, so that:
\begin{equation*}
t\mapsto \H_t \mbox{ admits an inverse } \H\mapsto t_\H.
\end{equation*}
In this situation, we reparameterize the family $f_{\H_t}$ using the parameter $t=t_\H$, thus writing $\dd f_t(x_t)=0$.

\begin{definition}\label{def:bif}
The instant $t_*\in [a,b]$ is a \emph{bifurcation instant} for the family $x_t$ of CMC embeddings if there exists a sequence $t_n\in [a,b]$ converging to $t_*$ and a sequence $y_n\in\Emb(N,M)$ converging to $x_{t_*}$, such that $y_n$ has CMC equal to that of $x_{t_n}$, i.e., $\dd f_{t_n}(y_n)=0$, but $y_n\neq x_{t_n}$.
\end{definition}

From the Implicit Function Theorem, see for instance \cite{bps2} or \cite[Prop 2.4]{ap}, \emph{degeneracy} of $x_{t_*}$ as a critical point of $f_{t_*}$ is a necessary condition for bifurcation. Nevertheless, it is in general not sufficient. A well-known sufficient criterion for bifurcation is a change in the Morse index $i(x_t)$ at the degeneracy instant $t_*$.

\begin{proposition}\label{prop:bifjumpmorseindex}
Let $(x_t)_{t\in[a,b]}$ be a $C^1$ family of embeddings of $N$ into $M$ having constant mean curvature $\H_t$, such that $x_t(N)=\partial\Omega_t$ is the boundary of a (bounded) open subset $\Omega_t\subset M$. Assume the following:
\begin{itemize}
\item[(i)] $\H_t'\ne0$ for all $t\in[a,b]$;
\item[(ii)] $x_a$ and $x_b$ are nondegenerate;
\item[(iii)] $i(x_a)\ne i(x_b)$.
\end{itemize}
Then, there exists a bifurcation instant $t_*\in \left]a,b\right[$ for the family $x_t$.
\end{proposition}

An abstract formulation and proof of the above criterion can be found, e.g., in  \cite[Thm II.7.3]{kielhofer}. For more details on how this criterion applies in the context of CMC embeddings, see \cite[Sec~5]{bps} or \cite[Sec~2, 3]{ap}.

\subsection{Bifurcation with symmetries}
In the bifurcation setup considered in our applications, the family $(x_t)_{t\in[a,b]}$ is invariant under various symmetry groups that we now describe. First of all, $x_t(\G/\Hs)$ are principal $\G$-orbits in $M$. Secondly, $x_t(\G/\Hs)$ are also invariant under a smaller isometry group, that we denote $\GG$, depending on which normality assumption is considered:
\begin{itemize}
\item[$\K\triangleleft\G$ :] $x_t(\G/\Hs)$ is invariant under the subaction of $\GG=\K$, by Lemma~\ref{lemma:normalityA};
\item[$\Hs\triangleleft\K$ :] $x_t(\G/\Hs)$ is invariant under the right action \eqref{eq:krightaction} of $\GG=\K/\Hs$, by Lemma~\ref{lemma:normalityB}.
\end{itemize}
The crucial feature of these $\GG$-actions on $x_t(\G/\Hs)$ is that their orbits are exactly the fibers of the homogeneous sphere bundles $q\colon x_t(\G/\Hs)\to \G/\K$, see \eqref{eq:sphbundle}. Our Delaunay-type hypersurfaces on $M$ are obtained as bifurcating solutions of the CMC variational problem in the space of $\GG$-invariant embeddings of $\G/\Hs$ in $M$, hence also have this $\GG$-symmetry. Geometrically, $\GG$-invariant embeddings of $\G/\Hs$ in $M$ are tubular graphs of functions $\phi\colon \G/\K\to\R$, i.e., these hypersurfaces of $M$ are the union of normal spheres to $S=\G/\K$ of radius $\phi(p)$ around each $p\in\G/\K$. If $\phi\equiv t$ is a constant function, then the corresponding hypersurface is the geodesic tube $x_t(\G/\Hs)$ of points in $M$ at distance $t$ from $S$.

In order to adapt the above bifurcation criterion (Proposition~\ref{prop:bifjumpmorseindex}) to this context with symmetries, consider again an abstract framework, consisting of a Riemannian manifold $M$ with an isometric $\GG$-action and a $C^1$ path $x_t$ of CMC embeddings \eqref{eq:xtabstract} that are $\GG$-invariant. In other words, $x_t$ is a path in the fixed point set $\Emb(N,M)^\GG$ of the natural $\GG$-action in the space $\Emb(N,M)$.
Each $x_t$ induces a $\GG$-action on the source manifold $N$, defined by $g\cdot p=x_t^{-1}\big(g\cdot x_t(p)\big)$, for all $g\in \GG$ and $p\in N$. In turn, this also defines a representation of $\GG$ on the space of functions on $N$. Using $\GG$-invariance of the functional $f_{\H_t}$, it is easy to see that for all $t\in[a,b]$ and all eigenvalues $\mu$ of the Jacobi operator $J_t:=J_{x_t}$, we have a representation of $\GG$ on the $\mu$-eigenspace $E^\mu_t$ of $J_{t}$. Denote by $(E^\mu_t)^\GG\subset E^\mu_t$ the subspace consisting of functions that are fixed by $\GG$, and set:
\begin{equation}\label{eq:defHtK}
(E^-_t)^\GG:=\bigoplus_{\mu<0} (E_t^\mu)^\GG.
\end{equation}
We define the \emph{$\GG$-Morse index of $x_t$}, denoted by $i^\GG(x_t)$, as:
\begin{equation}
i^\GG(x_t):=\dim\, (E^-_t)^\GG.
\end{equation}
In this context, a $\GG$-invariant CMC embedding $x_t\colon N\hookrightarrow M$ is \emph{$\GG$-nondegenerate} if
\begin{equation}\label{eq:equivnondeg}
(E_t^0)^\GG=\{0\},
\end{equation}
i.e., if there are no nontrivial $\GG$-invariant functions in the kernel of the Jacobi operator. Just like before, equivariant degeneracy of $x_t$ is a necessary (but not sufficient) condition for bifurcation, while a change in the Morse index is sufficient.

\begin{proposition}\label{prop:symmpreserv}
Let $(x_t)_{t\in[a,b]}$ be a $C^1$ family of $\GG$-invariant embeddings of $N$ into $M$ having constant mean curvature $\H_t$, such that $x_t(N)=\partial\Omega_t$ is the boundary of a (bounded) open $\GG$-invariant subset $\Omega_t\subset M$. Assume the following:
\begin{itemize}
\item[(i)] $\H_t'\ne0$ for all $t\in[a,b]$;
\item[(ii$^\GG$)] $x_a$ and $x_b$ are $\GG$-nondegenerate;
\item[(iii$^\GG$)] $i^\GG(x_a)\ne i^\GG(x_b)$.
\end{itemize}
Then, there exists a bifurcation instant $t_*\in \left]a,b\right[$ for the family $x_t$. Moreover, the corresponding bifurcating CMC embeddings are $\GG$-invariant.
\end{proposition}

\begin{proof}
The proof is obtained by applying Proposition~\ref{prop:bifjumpmorseindex} on the space of $\GG$-invariant embeddings $\Emb(N,M)^\GG$. By the Principle of Symmetric Criticality (see Palais \cite{palais}), critical points of the restriction of $f_t$ to $\Emb(N,M)^\GG$ are precisely the $\GG$-invariant embeddings of $N$ into $M$ with CMC equal to $t_\H$. Since isometries commute with the exponential map, it is easy to see that $\GG$-invariant functions on $N$ correspond to infinitesimal variations of $x_t$ through $\GG$-invariant embeddings. Assumption (ii$^\GG$) corresponds to nondegeneracy of $x_a$ and $x_b$ as critical points of the restrictions of $f_a$ and $f_b$ to $\Emb(N,M)^\GG$, while assumption (iii$^\GG$) corresponds to the change in the Morse index for the restriction of $f_t$ to $\Emb(N,M)^\GG$.
\end{proof}

\section{Main bifurcation result}\label{sec:mainpf}
We are now ready to prove our main result, that gives infinitely many bifurcation instants for the family of principal orbits on a cohomogeneity one manifold. 

\begin{theorem}\label{thm:main}
Let $M$ be a cohomogeneity one $\G$-manifold with principal isotropy $\Hs$. Let $S=\G/\K$ be a singular orbit that is not a fixed point and denote by $t_S$ the boundary point of $M/\G$ corresponding to $S$. Assume that the $\G$-invariant metric on $M$ is adapted near $S$ and that there exists a group $\GG$ that acts isometrically on $D(S)$, so that the $\GG$-orbits on a given principal orbit $x_t(\G/\Hs)$ are fibers of the sphere bundle \eqref{eq:sphbundle}.
Then, there is a converging sequence $t_q\to t_S$ of bifurcation instants for the family $x_t\colon \G/\Hs\hookrightarrow M$ of CMC embeddings of principal orbits. Moreover, partial symmetry preservation occurs at every $t_q$, in the sense that the corresponding bifurcating CMC embeddings are $\GG$-invariant but not $\G$-invariant.
\end{theorem}

\begin{proof}
Since $M$ is assumed compact and connected and the $\G$-action on $M$ has a singular orbit $S$, it follows that $M/\G$ has boundary and is hence a closed interval, see Subsection~\ref{subsec:topology}. Up to rescaling, assume\footnote{Along this proof, we normalize the closed interval $M/\G$ as $[0,1]$ instead of $[-1,1]$ in order to simplify notation, just like in the final subsections of Section~\ref{sec:cohom1}.} $M/\G=[0,1]$ and $t_S=0$. 

We apply Proposition~\ref{prop:symmpreserv} to the family \eqref{eq:porbits} of CMC embeddings $x_t\colon \G/\Hs\hookrightarrow M$ given by principal orbits. Every principal orbit $x_t(\G/\Hs)$ disconnects $M$ in two disjoint open bounded subsets \eqref{eq:decomp}. Thus, $x_t(\G/\Hs)=\partial \Omega_t$, where $\Omega_t$ is the component of $M\setminus x_t(\G/\Hs)$ containing $S$. Furthermore, denoting by $\H_t$ the mean curvature of $x_t$, it follows from \eqref{eq:sxtblowsup} that $\H_t'\ne0$, i.e., (i) is satisfied. To conclude the proof, we need to show the existence of decreasing sequences $a_q<b_q$ of positive real numbers, with $\lim_{q\to\infty}a_q=\lim_{q\to\infty}b_q=0$, for which assumptions (ii$^\GG$) and (iii$^\GG$) of Proposition~\ref{prop:symmpreserv} are satisfied. 

According to \eqref{eq:equivnondeg}, $t$ is a $\GG$-degeneracy instant if the kernel of the Jacobi operator $J_t$ contains nontrivial $\GG$-invariant functions on $\G/\Hs$. Observe that, by our hypothesis on the $\GG$-action on $\G/\Hs$, a function on $\G/\Hs$ is $\GG$-invariant precisely when it is the lift of a function on $S=\G/\K$ by the projection \eqref{eq:sphbundle}. Thus, from Proposition~\ref{prop:eigenlift}, $x_t$ is $\GG$-degenerate if and only if the constant\footnote{%
Notice that $\Ric(\vec n_{t})+\|\mathcal S_{t}\|^2$ is constant on $x_t(\G/\Hs)$ by homogeneity, for each $t$.}
\begin{equation}\label{eq:deginst}
\rho(t):=\Ric(\vec n_{t})+\|\mathcal S_{t}\|^2
\end{equation}
is a basic eigenvalue of $\Delta_t$. By \eqref{eq:inclusion}, this holds when $\rho(t)\alpha(t)^2\in\spec(\Delta_S)$.

Denote by $0<\lambda_1<\lambda_2<\ldots$ the eigenvalues of $\Delta_S$, counted with multiplicity. An embedding $x_t$ is $\GG$-nondegenerate if none of the $\lambda_i's$ are equal to $\rho(t)\alpha(t)^2$, and the $\GG$-Morse index $i^\GG(x_t)$ is equal to the number of $\lambda_i$'s, counted with multiplicity, that are less than $\rho(t)\alpha(t)^2$. Note that, in principle, not every $\GG$-degenerate instant produces a change in the Morse index, since the function $\rho(t)\alpha(t)^2$ may fail to be monotonic. Nevertheless, $\lim_{t\to0}\alpha(t)=1$ and $\lim_{t\to0}\rho(t)=+\infty$ by \eqref{eq:sxtblowsup}. As a result, the existence of $a_q$ and $b_q$ as above is obtained by elementary arguments, as follows.
Let $\tau>0$ be such that $\H_t'\neq0$, for all $0<t<\tau$, cf. \eqref{eq:sxtblowsup}.
 For $q\in\mathds N$, set
\begin{equation*}
C_q:=\big\{t\in\left]0,\tau\right[:\rho(t)\alpha(t)^2=\lambda_q\big\}.
\end{equation*}
Then:
\begin{itemize}
\item[(a)] for $q$ large enough, $C_q$ is a nonempty compact subset of $\left]0,\tau\right[$;
\item[(b)] $\max C_{q-1}<\min C_q\le\max C_q<\min C_{q+1}$, provided that $C_{q-1}\ne\emptyset$;
\item[(c)] the $C_q$'s accumulate at $t=0$ as $q\to\infty$, i.e., $\lim_{q\to\infty}\max C_q=0$.
\end{itemize}
Thus, for $q$ sufficiently large, choose
\[a_q\in\left]\max C_{q-1},\min C_q\right[ \quad\text{and}\quad b_q\in\left]\max C_q,\min C_{q+1}\right[.\]
It follows easily that $x_{a_q}$ and $x_{b_q}$ are $\GG$-nondegenerate, and $i^\GG(x_b)-i^\GG(x_a)>0$ is equal to the multiplicity of $\lambda_q$ as an eigenvalue of $\Delta_S$. Hence, Proposition~\ref{prop:symmpreserv} gives the existence of a bifurcation instant $t_q\in\left]a_q,b_q\right[$ and by (c), $\lim_{q\to\infty}t_q=0$. Break of $\G$-symmetry follows readily from the trivial observation that the only $\G$-invariant CMC hypersurfaces
in $M$ must be $\G$-orbits, all of which belong to the family $(x_t)_t$. Thus, every bifurcating branch consists of CMC embeddings that are $\GG$-invariant, but not $\G$-invariant.
\end{proof}

In view of Definition~\ref{def:bif} and Lemmas~\ref{lemma:normalityA} and \ref{lemma:normalityB}, applying the above result with $\GG=\K$ and $\GG=\K/\Hs$ proves the Theorem in the Introduction.

\section{Delaunay-type hypersurfaces in standard spheres}\label{sec:ex1}

The study of CMC hypersurfaces on spheres is a classic subject in differential geometry, that saw important developments in recent years, see e.g. \cite{ap,al,brendle,perdomo1,perdomo2}. A major contribution to the area was recently given by  Brendle~\cite{brendle} with a proof of the Lawson Conjecture, that states that the minimal Clifford torus
\begin{equation}\label{eq:cliffordtorus}
T_{\frac{\pi}{4}}:=\big\{(x_1,x_2,x_3,x_4)\in S^3\subset\R^4 : x_1^2+x_2^2=x_3^2+x_4^2=\tfrac12\big\}
\end{equation}
is the unique embedded minimal torus in $S^3$ endowed with the round metric, up to congruences. A glance at the above mentioned references reveals the crucial role played by symmetries. Shortly after the announcement of the proof of the Lawson Conjecture, Andrews and Li \cite{al} announced a full classification of embedded CMC tori in $S^3$, proving a conjecture of Pinkall and Sterling that asserts that any such embedding is rotationally symmetric, see also \cite{perdomo2}. This reduced the classification to a known case, see \cite{HyndParkMcCuan,perdomo1}.

The above minimal Clifford torus is an example of Delaunay surface. The classic Delaunay hypersurfaces in a round sphere $S^{n+1}$ are diffeomorphic to $S^k\times S^{n-k}$ and can be interpreted in terms of bifurcations from a family of orbits, as briefly described in \cite[\S 7]{pacard} and later explored in \cite{ap}. In this section, we first discuss these classical Delaunay hypersurfaces and reobtain some of them by using Theorem~\ref{thm:main}, and then explore the full generality of this result to produce \emph{new} Delaunay-type hypersurfaces on spheres, that are not diffeomorphic to the above.

\subsection{Clifford tori}\label{subsec:clifftori}
The minimal Clifford torus \eqref{eq:cliffordtorus} is an element of a $1$-parameter family of embedded CMC tori in $S^3$:
\begin{equation}\label{eq:cliffordtori}
T_{t}:=\big\{(x_1,x_2,x_3,x_4)\in S^3 : x_1^2+x_2^2=\cos^2 t, \, x_3^2+x_4^2=\sin^2 t \big\}, \quad 0<t<\tfrac{\pi}{2},
\end{equation}
which we call \emph{Clifford tori}. These tori are principal orbits of the standard action of the torus $\G=\mathsf T^2$ on $S^3\subset\C^2$, given by complex multiplication on each coordinate. This is a cohomogeneity one action with trivial principal isotropy $\Hs=1$ and singular isotropies $\K_-=1\times \mathsf S^1$ and $\K_+=\mathsf S^1\times 1$, corresponding to the endpoints of $S^3/\G=[0,\pi/2]$. Note the singular orbits $S_\pm=\G/\K_\pm$ are the geodesic circles to which $T_t$ collapse as $t\to 0$ or $t\to \pi/2$.
Since $\G=\mathsf T^2$ is abelian, both normality conditions are simultaneously satisfied. Also, it is easy to see that the round metric is adapted to both singular orbits $S_\pm$, with $\alpha(t)=\cos t$.
Thus, by the Theorem in the Introduction, we reobtain the existence of infinitely many rotationally symmetric CMC tori in $S^3$ not congruent to \eqref{eq:cliffordtori}.

The above conclusion matches the result announced by Andrews and Li \cite{al}, that \emph{every} embedded CMC torus in $S^3$ is rotationally invariant. The CMC tori that constitute the bifurcating branches of Delaunay surfaces are actually explicitly described in the classification results of \cite{HyndParkMcCuan,perdomo1}. Combining the results of \cite{al} and \cite{HyndParkMcCuan}, we get that such bifurcating CMC tori are of \emph{unduloid type} and have $\mathsf S^1\times\Z_m$ as isometry group, where the action of $\mathsf S^1\times\Z_m$ is exactly the restriction of the $\mathsf T^2$-action, i.e., the $\Z_m$ component acts by rotations on the directions orthogonal to the parallels.
Moreover, Perdomo~\cite{perdomo1} proved that for any $m\geq 2$, if $\H$ satisfies
\begin{equation}\label{eq:perdomos}
\cot\frac{\pi}{m}<\H<\frac{m^2-2}{2\sqrt{m^2-1}},
\end{equation}
then, there exists a compact embedded torus in $S^3$, different from any Clifford torus, with constant mean curvature $\H$ and isometry group $\mathsf S^1\times\Z_m$. Andrews and Li \cite{al} showed that there is \emph{at most one} such torus, up to congruences. These tori are exactly the bifurcating solutions reobtained above.

\begin{remark}
Combining the previous results, we have the remarkable consequence that the bifurcating solutions issuing from $x_{t_q^\pm}$ are increasingly symmetric as $q\to+\infty$. In fact, as $m\to+\infty$, the interval of $\H$'s that satisfy \eqref{eq:perdomos} goes to $+\infty$, which is the mean curvature of elements in bifurcating branches issuing from $x_{t_q^\pm}$ with $q\to+\infty$. This gives further insight on the partial symmetry preservation that takes place in this example; in that \emph{less} symmetry is lost at each bifurcating instant, as the principal orbits collapse to a singular one.
\end{remark}

\subsection{Higher dimensional Clifford tori}\label{subsec:highdimclifftori}
The cohomogeneity one action on $S^3$ with orbits \eqref{eq:cliffordtori} is a particular case of a cohomogeneity one action on  $S^{n+1}$, obtained as a \emph{sum action}. In other words, it is the restriction to the unit sphere of $\R^{n+2}$ of a representation with two irreducible factors with dimensions $k+1$ and ${n-k+1}$. The group diagram for such an action is
\begin{equation*}
\begin{gathered}
 \xymatrix@=4pt{& \SO(k+1)\SO(n-k+1) & \\
 & & \\
 \SO(k)\SO(n-k+1) \ar[uur] & &\SO(k+1)\SO(n-k) \ar[uul] \\
 & & \\
  & \SO(k)\SO(n-k)\ar[uur]\ar[uul] &}
\end{gathered}
\end{equation*}
where, again, the homomorphisms are the obvious block inclusions. Principal orbits of this action give a $1$-parameter family
\begin{equation}\label{eq:clifftori}
x_t\colon S^k\times S^{n-k}\to S^{n+1}, \quad 0<t<\tfrac\pi2,
\end{equation}
of CMC embeddings, that naturally generalizes \eqref{eq:cliffordtori}, see \cite[Ex.\ 6.50]{book}.

For this action, the only cases in which one of our normality assumptions is satisfied is when $k$ or $n-k$ is equal to $1$, corresponding respectively to $\K_-$ or $\K_+$ being normal in $\G$, giving principal orbits $\G/\Hs=S^1\times S^{n-1}$. Thus, our result implies existence of rotationally symmetric Delaunay-type hypersurfaces in $S^{n+1}$ diffeomorphic to $S^1\times S^{n-1}$ and not congruent to Clifford tori. These Delaunay-type hypersurfaces are part of larger classes discussed in Subsections~\ref{subsub:nonesssph} and \ref{subsec:kprod}.

Despite our result only applying in this subcase, it is known that for any choice of $k$ and $n$ there are infinitely many bifurcation values for the family \eqref{eq:clifftori} accumulating at both endpoints of $M/\G=[0,\pi/2]$. More precisely, there exist two converging sequences $t_q^\pm$ of bifurcation instants for the family \eqref{eq:clifftori},
\begin{equation*}
t_q^- = \arccos\sqrt{\frac{q(k+q+1)}{n-k+q(k+q+1)}}\quad \mbox{ and } \quad t_q^+ =\arccos\sqrt{\frac{k}{k+q(n-k+q+1)}},
\end{equation*}
where $t_q^-\to 0$ and $t_q^+\to\pi/2$. In particular, the Clifford tori $x_{t_q^\pm}$ are limits (in the Hausdorff distance) of pairwise noncongruent CMC embeddings of $S^k\times S^{n-k}$ in $S^{n+1}$, each of which is not congruent to any CMC Clifford torus. This was proved by  Al\'ias  and Piccione~\cite{ap}, through a direct computation of the Morse index showing that $i(x_t)\to+\infty$ as $t\to 0$ and $t\to\pi/2$, cf. Proposition~\ref{prop:bifjumpmorseindex}. A similar computation can be carried out in many cases where the normality assumptions fail, suggesting that our results should remain true under some weaker hypotheses.

\subsection{New Delaunay-type hypersurfaces}
Apart from Clifford tori, our results produce new families of Delaunay-type hypersurfaces on round spheres by using other cohomogeneity one actions.
These actions were classified by \cite{hl,straume}, see \cite[Tables E, F]{gwz} for a complete description. In order to simplify the classification scheme for such actions, one first considers \emph{essential} cohomogeneity one actions, i.e., actions so that no subaction is fixed point homogeneous and no normal subaction is orbit equivalent to it.

\subsubsection{Essential actions}
There are no essential cohomogeneity one actions on round spheres with a normal singular isotropy, and the only essential actions with $\Hs\triangleleft \K_-$ are given in the following table.

\begin{table}[htf]
\caption{Essential actions on spheres with $\Hs\triangleleft \K_-$.}\label{tab:essentialspheres}
\begin{tabular}{|ccccc|}
\hline
$M$ \rule{0pt}{2.5ex} \rule[-1.2ex]{0pt}{0pt} & $\G$ & $\K_-$ & $\K_+$ & $\Hs$ \\
\hline \hline
$S^{2k+3}$ \rule{0pt}{2.5ex} & $\mathsf{SO}(2)\mathsf{SO}(k+2)$ & $\Delta \mathsf{SO}(2)\mathsf{SO}(k)$ & $\mathds Z_2\cdot\mathsf{SO}(k+1)$ & $\mathds Z_2\cdot \mathsf{SO}(k)$ \\
$S^{15}$ & $\mathsf{SO}(2)\mathsf{Spin}(7)$ & $\Delta \mathsf{SO}(2)\mathsf{SU}(3)$ & $\mathds Z_2\cdot\mathsf{Spin}(6)$ & $\mathds Z_2\cdot\mathsf{SU}(3)$ \\
$S^{13}$\rule[-1.2ex]{0pt}{0pt} & $\mathsf{SO}(2)\mathsf G_2$ & $\Delta \mathsf{SO}(2)\mathsf{SU}(2)$ & $\mathds Z_2\cdot\mathsf{SU}(3)$ & $\mathds Z_2\cdot\mathsf{SU}(2)$ \\
\hline
\end{tabular}
\end{table}

\begin{remark}
In all tables we use the following notation conventions. If $G_1$ and $G_2$ are Lie groups, then $G_1G_2$ denotes their direct product, and $G_1\cdot G_2$ denotes a quotient of $G_1G_2$ by a finite central subgroup. Moreover, if $H$ is a subgroup of $G_1$ and of $G_2$, then $\Delta H$ denotes the diagonal embedding of $H$ in $G_1G_2$.
\end{remark}

In all the above actions on spheres, $\K_-$ is connected and $\K_-/\Hs=S^1$, so $\Hs\triangleleft\K_-$ by Corollary~\ref{cor:normality1conn}. Furthermore, 
the round metric is invariant under the corresponding right circle actions by $\GG=\K_-/\Hs$ in a neighborhood of $\G/\K_-$.\footnote{This can be directly verified for the case of $S^{2k+3}$, using that the $\SO(2)\SO(k+2)$-action comes from an exterior tensor product of representations of $\SO(2)$ in $\R^2$ and $\SO(k+2)$ in $\R^{k+2}$. Since the other two actions are given as certain subactions of this first case when $k=6$ and $k=5$, it follows immediately that also these $\K_-/\Hs$-actions are isometric. The moduli space of (adapted) metrics in the latter cases is of course strictly larger (since the isometry group is smaller) and existence of Delaunay-type hypersurfaces in $S^{15}$ and $S^{13}$ equipped with these metrics does not follow from the actions being subactions of $\SO(2)\SO(k+2)$ for appropriate $k$.} Thus, we get existence of rotationally symmetric Delaunay-type hypersurfaces whose double covering is a product of a circle and a unit tangent bundle of a lower dimensional sphere, namely $S^1\times T_1 S^{k+1}$, $S^1\times T_1 S^7$ and $S^1\times T_1 S^6$ corresponding to $S^{2k+3}$, $S^{15}$ and $S^{13}$, respectively. We stress that this \emph{rotational symmetry} of hypersurfaces in $S^n$ is not in the usual sense, i.e., there does not exist an axis in $\R^{n+1}$ with respect to which these hypersurfaces in $S^n\subset\R^{n+1}$ are rotationally symmetric. Instead, the rotational symmetry is with respect to the free isometric circle action on $S^n\setminus S_+$ described in Lemma~\ref{lemma:normalityB}, that does not extend to the whole $S^n$ in these cases.

\subsubsection{Nonessential actions}\label{subsub:nonesssph}
It is easy to check that the only nonessential cohomogeneity one actions on spheres that satisfy one of the normality conditions are sum actions coming from representations where one of the irreducible summands has dimension $2$ or $4$. More precisely, the possible actions are the $\SO(k+1)\SO(n-k+1)$-action described in Subsection~\ref{subsec:highdimclifftori} with either $k$ or $n-k$ equal to $1$, together with complex and quaternionic versions. In short, we have the following:

\begin{table}[htf]
\caption{Nonessential actions on spheres with $\Hs\triangleleft \K_-$.}\label{tab:nonessentialspheres}
\begin{tabular}{|ccccc|}
\hline
$M$ \rule{0pt}{2.5ex} \rule[-1.2ex]{0pt}{0pt} & $\G$ & $\K_-$ & $\K_+$ & $\Hs$ \\
\hline \hline
$S^{k+1}$ \rule{0pt}{2.5ex} & $\mathsf{SO}(2)\mathsf{SO}(k)$ & $\mathsf{SO}(2)\mathsf{SO}(k-1)$ & $\mathsf{SO}(k)$ & $\mathsf{SO}(k-1)$ \\
$S^{2k+1}$ & $\mathsf{U}(1)\mathsf{U}(k)$ & $\mathsf{U}(1)\mathsf{U}(k-1)$ & $\mathsf{U}(k)$ & $\mathsf{U}(k-1)$ \\
$S^{4k+3}$\rule[-1.2ex]{0pt}{0pt} & $\mathsf{Sp}(1)\Sp(k)$ & $\mathsf{Sp}(1)\mathsf{Sp}(k-1)$ & $\mathsf{Sp}(k)$ & $\mathsf{Sp}(k-1)$ \\
\hline
\end{tabular}
\end{table}
Thus, the corresponding Delaunay-type hypersurfaces are diffeomorphic to products of spheres $S^1\times S^{k-1}$, $S^1\times S^{2k-1}$ and $S^3\times S^{4k-1}$ corresponding to $S^{k+1}$, $S^{2k+1}$ and $S^{4k+3}$, respectively. These are diffeomorphic to previously mentioned higher dimensional version of Clifford tori in Subsection~\ref{subsec:highdimclifftori}, however the symmetry groups are different. In particular, this result produces Delaunay-type hypersurfaces on spheres equipped with different invariant metrics, not only the round one.

\section{Delaunay-type hypersurfaces in projective spaces}\label{sec:ex2}

Apart from constructing new topological types of Delaunay-type hypersurfaces in $S^{n+1}$, our methods also yield existence of Delaunay-type hypersurfaces in other compact rank one symmetric spaces. Cohomogeneity one actions on $\C P^n$, $\Hr P^n$ and $\Ca P^2$ are classified, and a detailed description of such actions and their isotropy groups can be found in \cite{gwz}.
We now describe which of these cohomogeneity one actions satisfy our hypotheses and hence produce Delaunay-type hypersurfaces.

We start by noting that no such actions occur on $\Ca P^2$, so we focus our attention to the projective spaces $\C P^n$ and $\Hr P^n$. If $\G$ acts with cohomogeneity one on $\C P^n$ (respectively $\Hr P^n$), then there exists a cohomogeneity one action of $\widetilde \G$ on $S^{2n+1}$ (respectively $S^{4n+3}$) and a normal subgroup $\mathsf N=\U(1)$ (respectively $\mathsf N=\Sp(1)$) of $\widetilde \G$ whose induced action is a Hopf action, so that $\widetilde \G/\mathsf N=\G$. This means that a full classification of the cohomogeneity one actions on $\C P^n$ and $\Hr P^n$ that satisfy our hypotheses can be directly obtained from the corresponding classification for actions on spheres, given in the previous section. Furthermore, this implies that the Fubini-Study metric in $\C P^n$ and $\Hr P^n$ is adapted with respect to such cohomogeneity one $\G$-actions, since the round metric is adapted with respect to the lifted cohomogeneity one $\widetilde\G$-actions.

\subsection{Essential actions}
Essential actions on projective spaces that satisfy one of the normality assumptions are obtained from Table~\ref{tab:essentialspheres}, resulting in the following:

\begin{table}[htf]
\caption{Essential actions on projective spaces with $\Hs\triangleleft \K_-$.}\label{tab:essproj}
\begin{tabular}{|ccccc|}
\hline
$M$ \rule{0pt}{2.5ex} \rule[-1.2ex]{0pt}{0pt} & $\G$ & $\K_-$ & $\K_+$ & $\Hs$ \\
\hline \hline
$\C P^{k+1}$ \rule{0pt}{2.5ex} & $\mathsf{SO}(k+2)$ & $\mathsf{SO}(2)\mathsf{SO}(k)$ & $\mathsf{O}(k+1)$ & $\mathds Z_2\cdot \mathsf{SO}(k)$ \\
$\C P^{7}$ & $\mathsf{Spin}(7)$ & $ \mathsf{SO}(2)\mathsf{SU}(3)$ & $\mathds Z_2\cdot\mathsf{Spin}(6)$ & $\mathds Z_2\cdot\mathsf{SU}(3)$ \\
$\C P^{6}$ \rule[-1.2ex]{0pt}{0pt}  & $\mathsf G_2$ & $\mathsf{U}(2)$ & $\mathds Z_2\cdot\mathsf{SU}(3)$ & $\mathds Z_2\cdot\mathsf{SU}(2)$ \\
\hline
\end{tabular}
\end{table}

From the above, we have existence of rotationally symmetric Delaunay-type hypersurfaces whose double covering is a unit tangent bundle of a lower dimensional sphere, namely $T_1 S^{k+1}$, $T_1 S^7$ and $T_1 S^6$ corresponding to $\C P^{k+1}$, $\C P^7$ and $\C P^6$, respectively. Again, rotational symmetry is to be interpreted as invariance under the isometric right $\K_-/\Hs$-action in $D(S_-)$.

\subsection{Nonessential actions}
Nonessential actions on $\C P^k$ and $\Hr P^k$ lift to sum actions on $S^{2k+1}$ and $S^{4k+3}$, respectively. It is easy to see that the only actions on projective spaces satisfying one of the normality conditions are those that lift to one of the actions in Table~\ref{tab:nonessentialspheres}. The latter commute with the Hopf actions that define the corresponding projective spaces, and the low dimensional irreducible summand is always the restriction of this Hopf action to that subspace. Consequently, in the quotient, the singular orbit corresponding to that subspace becomes a fixed point, while the other singular orbit becomes its cut locus, see \cite[Ex.\ 6.52, 6.54]{book}. For completeness, we list the corresponding groups in Table~\ref{tab:nonessproj}.

\begin{table}[htf]
\caption{Nonessential actions on projective spaces with $\Hs\triangleleft \K_-$.}\label{tab:nonessproj}
\begin{tabular}{|ccccc|}
\hline
$M$ \rule{0pt}{2.5ex} \rule[-1.2ex]{0pt}{0pt} & $\G$ & $\K_-$ & $\K_+$ & $\Hs$ \\
\hline \hline
$\C P^{k}$ \rule{0pt}{2.5ex} & $\mathsf{U}(k)$ & $\mathsf{U}(1)\mathsf{U}(k-1)$ & $\mathsf{U}(k)$ & $\mathsf{U}(k-1)$ \\
$\Hr P^{k}$ \rule[-1.2ex]{0pt}{0pt}  & $\Sp(k)$ & $\Sp(1)\Sp(k-1)$ & $\Sp(k)$ & $ \Sp(k-1)$ \\
\hline
\end{tabular}
\end{table}

\begin{example}\label{ex:cpn}
The action of $\G=\U(k)$ on $\C P^k$ lifts to a sum action of $\widetilde\G=\U(1)\U(k)$ on $S^{2k+1}$, i.e., the restriction to the sphere of a reducible representation $\C\oplus\C^k$. The singular orbits of $\widetilde \G$ on $S^{2k+1}$ are the unit spheres of $\C$ and $\C^k$, i.e., $S^1$ and $S^{2k-1}$. After projecting to $\C P^k$, these become the singular orbits of the $\U(k)$-action, which are hence a fixed point $S_+=\{p\}$ and its cut locus $S_-=\C P^{k-1}$. The principal $\U(k)$-orbits are distance spheres $S^{2k-1}$ centered at $p$, and the homogeneous fibration \eqref{eq:homfib} is the Hopf fibration $S^1\to S^{2k-1}\to\C P^{k-1}$. Metrically, these distance spheres are equipped with Berger metrics, the size of the Hopf fiber being proportional to the distance to $p$, i.e., the radius of the sphere.

In this context, it is also easy to describe the $\K_-/\Hs$-action on the neighborhood of $S_-$. Namely, if we consider homogeneous coordinates with $p=[1:0:\dots : 0]\in \C P^k$ and $S_-=\{[0:*:\dots:*]\}$, then $\K_-/\Hs=\U(1)$ acts by rotating the first coordinate. This action is clearly isometric in $D(S_-)$, fixes $S_-$ (which recovers the fact that $S_-=\C P^{k-1}$ is totally geodesic in $\C P^k$, cf. Remark~\ref{rem:totgeod}) and its orbits are normal circles to $S_-$, i.e., fibers of the Hopf fibration $S^1\to S^{2k-1}\to\C P^{k-1}$. This action does not extend smoothly to $p$, corresponding to $\Hs\ntriangleleft\K_+$. A geometric explanation for this is that the corresponding $\U(1)$-action field (which is a Killing field in $\C P^k$) has increasingly large norm as the distance to $S_-=\C P^{k-1}$ increases (i.e., as the distance to $S_+=\{p\}$ decreases). As a result, this Killing field on $\C P^k\setminus\{p\}$ cannot be extended to $p$, since this would only be possible if $p$ was a zero of this field.

In the above context, our result gives existence of Delaunay-type spheres $S^{2k-1}$ in $\C P^k$, that are rotationally symmetric with respect to the above $\U(1)$-action. Analogously, one gets Delaunay-type spheres $S^{4n-1}$ in $\Hr P^k$ invariant under $\Sp(1)$.
\end{example}

\begin{remark}
A natural concern is if the Delaunay-type hypersurfaces bifurcating from distance spheres centered at $p\in\C P^k$ are legitimally new CMC hypersurfaces, or if they could be obtained via other isometries of $\C P^k$. For instance, there could be a sequence of points $p_n\in\C P^k$, whose distance spheres of appropriate radii accumulate on distance spheres centered at $p$. Note that, in particular, this would force the sequence $\{p_n\}$ to converge to $p$. This situation can be easily excluded using the $\K_-/\Hs$-invariance of the bifurcation solutions. If there were such a sequence $\{p_n\}$ with $\K_-/\Hs$-invariant distance spheres, then $p_n$ would be fixed by $\K_-/\Hs$. However, the only fixed points of $\K_-/\Hs$ in $D(S_-)$ lie in $S_-$, hence $\{p_n\}$ are at a bounded distance away from $p$ and cannot possibly converge to $p$.
\end{remark}

\section{Delaunay-type hypersurfaces in Kervaire spheres}\label{sec:ex3}

Consider the Brieskhorn variety $M^{2n-1}_d\subset \mathds C^{n+1}$ defined by the equations
\begin{equation*}
\begin{cases}
z_0^d+z_1^2+\cdots+z_n^2=0, \\
|z_0|^2+|z_1|^2+\cdots +|z_n|^2=1.
\end{cases}
\end{equation*}
When $n$ and $d$ are odd, $M^{2n-1}_d$ is homeomorphic to the sphere $S^{2n-1}$. Nevertheless, if $2n-1\equiv 1 \mod 8$, then $M^{2n-1}_d$ is not diffeomorphic to $S^{2n-1}$, and such manifolds are called \emph{Kervaire exotic spheres}.

According to \cite{gvwz}, it was first observed by Calabi in dimension $5$ and later in \cite{hh} that $M^{2n-1}_d$ carries a cohomogeneity one action of $\SO(2)\SO(n)$, given by $(e^{i\theta},A)\cdot(z_0,\dots,z_n)=\big(e^{2i\theta}z_0,A(z_1,\dots,z_n)\big)$. The group diagram for such action is
\begin{equation}\label{eq:brieskhorn}
\begin{gathered}
 \xymatrix@=4pt{& \SO(2)\SO(n) & \\
 & & \\
\SO(2)\SO(n-2) \ar[uur] & &\mathsf O(n-1) \ar[uul] \\
 & & \\
  & \Z_2\SO(n-2)\ar[uur]\ar[uul] &}
\end{gathered}
\end{equation}
and it can be easily seen that $\Hs\triangleleft\K_-$ from the embedding\footnote{Alternatively, one can also see this as a consequence of Corollary~\ref{cor:normality1conn} provided $n\geq 4$.} of $\Hs$ into $\K_-$. Since there is no canonical invariant metric on $M^{2n-1}_d$, simply choose any metric that is adapted near $S_-$, see Proposition~\ref{prop:alwaysadapted}. Then, our result implies existence of infinitely many Delaunay-type hypersurfaces on $M^{2n-1}_d$, near $S_-$, diffeomorphic (up to double covering) to $S^1\times T_1 S^{n-1}$ and that are invariant under a free right isometric circle action. In particular, this provides examples of Delaunay-type rotationally symmetric hypersurfaces in Kervaire exotic spheres.

\section{Delaunay-type hypersurfaces in low dimensions}\label{sec:ex4}

Cohomogeneity one actions on simply-connected compact manifolds have been classified up to dimension $7$ by Hoelscher~\cite{hoelschert,hoelscher}. 
We now analyze this classification, verifying in which cases one of the normality assumptions is satisfied. In all such cases, Delaunay-type hypersurfaces are then obtained for every choice of adapted metric. This provides an exhaustive list of Delaunay-type hypersurfaces in low dimensional manifolds that can be obtained from our result.

It is easy to see that the only cohomogeneity one action on a compact simply-connected $2$-dimensional manifold is the rotation action on $S^2$, all of whose singular orbits are fixed points. Thus, we start our analysis in dimension $3$.

\subsection{3D examples}
If $\dim M=3$, then the possible groups $\G$ that act with cohomogeneity one on $M$ have dimension $2$ or $3$. If $\dim \G=2$, then $\G$ must be a torus acting on $S^3$ as described in Subsection~\ref{subsec:clifftori}. In this case, as mentioned above, we reobtain the classical rotationally symmetric Delaunay-type tori in $S^3$. If $\dim \G=3$, then $\G=\SU(2)$ (up to covering) and the action is a rotation action on $S^3$, so both singular orbits are fixed points.

\subsection{4D examples}
If $\dim M=4$, then the possible groups $\G$ that act with cohomogeneity one on $M$ have dimensions ranging from $3$ to $6$. An analysis of which actions satisfy our normality assumptions is given in Table~\ref{tab:hoelscher}.

\begin{table}[htf]
\caption{Cohomogeneity one actions on simply-connected $4$-manifolds with $\Hs\triangleleft\K_-$, extracted from \cite{hoelschert,parker}.}\label{tab:hoelscher}
\begin{tabular}{|ccccc|}
\hline
$M$ \rule{0pt}{2.5ex} \rule[-1.2ex]{0pt}{0pt} & $\G$ & $\K_-$ & $\K_+$ & $\Hs$ \\
\hline \hline
$S^2\times S^2$ \rule{0pt}{2.5ex} & $\mathsf S^3\times\mathsf S^1$ & $\mathsf S^1\times\mathsf S^1$ & $\mathsf S^1\times\mathsf S^1$  & $\mathsf S^1\times 1$ \\
$S^4$ & $\mathsf S^3\times\mathsf S^1$ & $\mathsf S^1\times\mathsf S^1$ & $\mathsf S^3\times\mathsf S^1$  & $\mathsf S^1\times 1$ \\
$S^2\times S^2$ & $\mathsf S^3$ & $\mathsf S^1$ & $\mathsf S^1$  & $\Z_{2n}$ \\
$\C P^2\# \overline{\C P}^2$ & $\mathsf S^3$ & $\mathsf S^1$ & $\mathsf S^1$  & $\Z_{2n+1}$ \\
$\C P^2$  \rule[-1.2ex]{0pt}{0pt} & $\mathsf S^3$ & $\{e^{i\theta}\}$ & $\{\{e^{j\theta}\}\cup\{ie^{j\theta}\}\}$ & $\langle i \rangle$ \\
\hline
\end{tabular}
\end{table}

In all the cases, the Theorem in the Introduction applies near $S_-=\G/\K_-$ and, in some cases, it also applies near $S_+=\G/\K_+$, e.g., when both $\K_\pm$ are abelian. In this way, we get Delaunay-type hypersurfaces in $S^4$, $\C P^2$, $S^2\times S^2$, and $\C P^2\# \overline{\C P}^2$, according to Table~\ref{tab:hoelscher}. There are also many cases of $4$-dimensional non-simply-connected cohomogeneity one manifolds to which our results apply. This general classification has been carried out by Parker~\cite{parker}, and one can easily identify many cases where our normality assumptions are satisfied.

\subsection{5D, 6D, and 7D examples}
In order to describe these examples, we say that a cohomogeneity one action is \emph{reducible} if there is a proper normal subgroup of $\G$ that acts by cohomogeneity one with the same orbits. According to Hoelscher~\cite{hoelscher}, the nonreducible cohomogeneity one action on a compact connected manifold of dimension $5$, $6$ or $7$ by a compact connected group are:
\begin{itemize}
\item isometric actions on a symmetric space;
\item product actions;
\item $\SO(2)\SO(n)$-actions on the Brieskhorn varieties $M_d^{2n-1}$ given by \eqref{eq:brieskhorn};
\item actions listed in \cite[Table I, II]{hoelscher}.
\end{itemize}
Many of the above possibilities already appeared in previous sections, e.g., the case of $M^{2n-1}_d$. Regarding the exceptional actions found in \cite[Table I, II]{hoelscher}, we list the primitive actions that satisfy our normality conditions in Table~\ref{tab:hoelscherthm}.

\begin{table}[htf]
\caption{Primitive actions in dimensions $5$, $6$ and $7$ with $\Hs\triangleleft\K_-$, extracted from \cite[Table I]{hoelscher}.}\label{tab:hoelscherthm}
\begin{tabular}{|ccccc|}
\hline
$M$ \rule{0pt}{2.5ex} \rule[-1.2ex]{0pt}{0pt} & $\G$ & $\K_-$ & $\K_+$ & $\Hs$ \\
\hline \hline
$M^5_d$ \rule{0pt}{2.5ex} & $\mathsf S^3\times\mathsf S^1$ & $\{(e^{jp\theta},e^{i\theta})\}$ & $\{(e^{i\theta},1)\}\cdot \Hs$  & $\langle(j,i)\rangle$ \\
\multicolumn{5}{|r|}{where $p\equiv1\mod4$} \\ \hline
$M^7_{6c}$  \rule{0pt}{2.5ex} & $\mathsf S^3\times\mathsf S^3$ & $\{(e^{ip_-\theta},e^{iq_-\theta})\}$ & $\{(e^{jp_+\theta},e^{jq_+\theta})\}\cdot \Hs$ & $\langle(i,i)\rangle$ \\
\multicolumn{5}{|r|}{where $p_-,q_-\equiv1\mod4$} \\ \hline
$M^7_{6d}$  \rule{0pt}{2.5ex} & $\mathsf S^3\times\mathsf S^3$ & $\{(e^{ip_-\theta},e^{iq_-\theta})\}\cdot \Hs$ & $\{(e^{jp_+\theta},e^{jq_+\theta})\}\cdot \Hs$ & $\langle(i,i),(1,-1)\rangle$ \\
\multicolumn{5}{|r|}{where $p_-,q_-\equiv1\mod4$, $p_+$ even} \\ \hline
$M^7_{6g}$  \rule{0pt}{2.5ex} \rule[-1.2ex]{0pt}{0pt} & $\mathsf S^3\times\mathsf S^3$ & $\{(e^{ip\theta},e^{iq\theta})\}$ & $\Delta S^3\cdot \Z_n$ & $\Z_n$ \\
\multicolumn{5}{|r|}{where $n=2$ and $p$ or $q$ even, or $n=1$ and $p,q$ arbitrary} \\
\hline
\end{tabular}
\end{table}

Our result applies to all such examples, when endowed with adapted metrics. On all these actions, $\Hs\triangleleft\K_-$ because either $\K_-$ is abelian or due to Proposition~\ref{prop:HnormalinK1}. Moreover, we remark that in the last example $M^7_{6g}$, we also have $\Hs\triangleleft\K_+$.

\section{Other constructions}\label{sec:ex5}

\subsection{Extensions}
A standard construction in cohomogeneity one is to extend the $\G$-action to a larger group that also acts by cohomogeneity one as follows.

\begin{lemma}
Let $M$ be a cohomogeneity one manifold with group diagram $\Hs\subset\{\K_-,\K_+\}\subset\G$. For any Lie group extension $\G\hookrightarrow\widetilde\G$, the group diagram $\Hs\subset\{\K_-,\K_+\}\subset\widetilde\G$ corresponds to a cohomogeneity one manifold $\widetilde M$ that is the total space of a fiber bundle $M\to \widetilde M\to \widetilde\G/\G$.
\end{lemma}

Note that this extension process increases the codimension of singular orbits, but the isotropy groups remain unchanged. Thus, if the Theorem in the Introduction applies with the assumption $\Hs\triangleleft\K_-$ or $\Hs\triangleleft\K_+$ to $M$, then it automatically applies to $\widetilde M$ with the same assumption. Using this in the above concrete examples one obtains a very large class of Delaunay-type hypersurfaces of arbitrary dimension.

\subsection{Spherical pairs}\label{subsec:kprod}
Another large class of cohomogeneity one manifolds that satisfy the normality assumption $\Hs\triangleleft\K_-$ can be obtained by using the well-known table of transitive actions on spheres, that we reproduce below for the readers' convenience. Choose $\K_-$ and $\Hs$ from this table, so that $\K_-/\Hs=S^n$, and set $\K_+:=\Hs\times\mathsf S^1$ (or $\K_+:=\Hs\times\mathsf S^3$). Then, for any $\G$ that contains $\K_\pm$, one has a cohomogeneity one manifold given by the group diagram $\Hs\subset\{\K_-,\K_+\}\subset \G$ to which our result applies, with $\GG=\K_+/\Hs$. For example, one can always choose $\G=\K_-\times\mathsf S^1$ (or $\G=\K_-\times\mathsf S^3$). In this case, we get an action on $M=S^{n+2}$ (or $M=S^{n+4}$) whose singular orbits are $S_-=S^1$ (or $S_+=S^3$) and principal orbits are $S^1\times S^n$ (or $S^3\times S^n$). These are precisely the sum actions on spheres that give Clifford tori as the corresponding Delaunay-type hypersurfaces discussed in Subsection~\ref{subsub:nonesssph}. Nevertheless, one can choose a different group $\G$ to obtain different examples.

\begin{table}[htf]
\caption{Transitive actions on $S^n$}
\begin{tabular}{|ccc|}
\hline
$M$ \rule{0pt}{2.5ex} \rule[-1.2ex]{0pt}{0pt} & $\K$ & $\Hs$ \\
\hline \hline
$S^n$ \rule{0pt}{2.5ex} & $\SO(n+1)$ & $\SO(n)$ \\
$S^{2n+1}$  & $\SU(n+1)$ & $\SU(n)$ \\
$S^{2n+1}$  & $\U(n+1)$ & $\U(n)$ \\
$S^{4n+3}$  \rule[-1.2ex]{0pt}{0pt} & $\Sp(n+1)$ & $\Sp(n)$ \\
\hline
\end{tabular}\quad
\begin{tabular}{|ccc|}
\hline
$M$ \rule{0pt}{2.5ex} \rule[-1.2ex]{0pt}{0pt} & $\K$ & $\Hs$ \\
\hline \hline
$S^{4n+3}$ \rule{0pt}{2.5ex} & $\Sp(n+1)\Sp(1)$ & $\Sp(n)\Delta\Sp(1)$ \\
$S^{4n+3}$ & $\Sp(n+1)\U(1)$ & $\Sp(n)\Delta\U(1)$ \\
$S^{15}$ & $\Spin(9)$ & $\Spin(7)$ \\
$S^7$  & $\Spin(7)$ & $\mathsf G_2$ \\
$S^6$  \rule[-1.2ex]{0pt}{0pt} & $\mathsf G_2$ & $\SU(3)$ \\
\hline
\end{tabular}
\end{table}

\subsection{Doubles}
Suppose $M$ is a cohomogeneity one manifold that satisfies the assumptions of the Theorem in the Introduction for the singular orbit $S_-=\G/\K_-$. Regardless of what the other singular isotropy $\K_+$ is, we can define a new cohomogeneity one manifold by setting $\K_+:=\K_-$. We call this procedure a \emph{double}, since the resulting manifold is equivariantly diffeomorphic to the disk bundle $D(S_-)$ glued with a copy of itself along the boundary. Since the assumptions were satisfied near $S_-$, they are satisfied near both $S_\pm$ in the new manifold. This procedure can clearly change the topological type of $M$, but the tubular neighborhoods $D(S_-)$ in the original manifold and the resulting manifold are isometric.

\begin{example}
As a concrete example of the above procedure, consider the $\U(k)$-action on $\C P^k$ with a fixed point $S_+=\{p\}$, described in Example~\ref{ex:cpn}. Then $\Hs=\U(k-1)$ sits in $\K_-=\U(1)\U(k-1)$ as a factor and is hence a normal subgroup, but $\Hs$ is not normal in $\K_+=\U(k)$. Replacing $\K_+$ with $\K'_+:=\K_-=\U(1)\U(k-1)$, we get a new cohomogeneity one manifold $M$ where both singular orbits are $S'_\pm=\C P^{k-1}$ and the Theorem in the Introduction applies to both of them. It is easy to see that $M\cong\C P^k\#\overline{\C P}^k$, e.g., by using a metric ball around $p$ to be the deleted $k$-ball needed for the connected sum operation. The boundary $S^{2k-1}$ of such metric ball is a principal orbit of the original $\U(k)$-action and is the interface along which the two disk bundles $D(S_-)=\G\times_{\K_-} D$ are glued together. Thus, our result applied in this setup gives existence of rotationally symmetric Delaunay-type hypersurfaces $S^{2k-1}$ in $\C P^k\#\overline{\C P}^k$.
\end{example}

\end{document}